\documentclass[onefignum,onetabnum]{siamart190516}
\usepackage[utf8]{inputenc}
 
\usepackage[english]{babel}
\usepackage{graphicx}
\usepackage{fancyhdr}

\usepackage{amsfonts}
\usepackage{amssymb}

\usepackage{amsmath}
\usepackage{amsfonts}
\usepackage{amssymb}
\usepackage{mathtools}
\usepackage{enumitem}
\usepackage[percent]{overpic}
\usepackage{tikz}
\usepackage{mathrsfs}
\usepackage{xargs}
\usepackage[colorinlistoftodos,prependcaption,textsize=tiny]{todonotes}
\newcommandx{\at}[2][1=]{\todo[linecolor=red,backgroundcolor=red!25,bordercolor=red,#1]{#2}}

\usepackage{algorithm}
\usepackage{algorithmic}

\usepackage{mathtools,booktabs}

\usepackage{varwidth}
\usepackage{hyperref}
\usepackage[capitalize]{cleveref}
\usepackage{cite}

\definecolor{rev1}{HTML}{cb270f}
\definecolor{rev2}{HTML}{1c8235}

\title{Computing semigroups with error control\thanks{\funding{This work was supported by a Research Fellowship at Trinity College, Cambridge.}}}
\author{Matthew J. Colbrook\thanks{Department of Applied Mathematics and Theoretical Physics, University of Cambridge. (\email{m.colbrook@damtp.cam.ac.uk})}}

\begin{document}

\date{}

\maketitle

\begin{abstract}
We develop an algorithm that computes strongly continuous semigroups on infinite-dimensional Hilbert spaces with explicit error control. Given a generator $A$, a time $t>0$, an arbitrary initial vector $u_0$ and an error tolerance $\epsilon>0$, the algorithm computes $\exp(tA)u_0$ with error bounded by $\epsilon$. The algorithm is based on a combination of a regularized functional calculus, suitable contour quadrature rules, and the adaptive computation of resolvents in infinite dimensions. As a particular case, we show that it is possible, even when only allowing pointwise evaluation of coefficients, to compute, with error control, semigroups on the unbounded domain $L^2(\mathbb{R}^d)$ that are generated by partial differential operators with polynomially bounded coefficients of locally bounded total variation. For analytic semigroups (and more general Laplace transform inversion), we provide a quadrature rule whose error decreases like $\exp(-cN/\log(N))$ for $N$ quadrature points, that remains stable as $N\rightarrow\infty$, and which is also suitable for infinite-dimensional operators. Numerical examples are given, including: Schr\"odinger and wave equations on the aperiodic Ammann--Beenker tiling, complex perturbed fractional diffusion equations on $L^2(\mathbb{R})$, and damped Euler--Bernoulli beam equations.

\end{abstract}



\vspace{-1mm}

\section{Introduction}\label{sec_intro_sec}

Given a linear operator $\! A$ on an \textit{infinite-dimensional} separable Hilbert space $\mathcal{H}$, can we numerically compute, with \textit{error control}, the solution of
\begin{equation}\setlength\abovedisplayskip{6pt}\setlength\belowdisplayskip{6pt}
\label{cauchy_prob}
u'(t)=Au(t)\text{ for }t\geq 0,\quad\text{with initial condition}\quad u(0)=u_0\in \mathcal{H}?
\end{equation}
The desired solution is written as $u(t)=\exp(tA)u_0$ and made rigorous through the theory of semigroups \cite{pazy2012semigroups,arendt2001cauchy}. Equation \eqref{cauchy_prob} arises in numerous applications and there exist many numerical methods designed to approximate $u(t)$, including but not limited to:
contour methods (the method adopted in the current paper) \cite{weideman2007parabolic,trefethen2014exponentially,sheen2003parallel,hale2008computing}; domain truncation and absorbing boundary conditions (e.g., when $A$ represents a differential operator on an unbounded domain) \cite{engquist1977absorbing,arnold2003discrete,tsynkov1998numerical,szeftel2004design,antoine2008review}; Galerkin methods \cite{lubich_qm_book,kormann2019stable,lasser2020computing}; Krylov methods \cite{grimm2012resolvent,gockler2013convergence,liesen2013krylov};
rational approximations \cite{crouzeix1993stability,brenner1979rational,palencia1993stability}; and series expansions, splitting methods, and exponential integrators \cite{higham2005scaling,lubich2008splitting,iserles2018magnus,hochbruck2010exponential,mclachlan2002splitting,al2011computing}. 

The majority of convergence results in the literature concern specific cases of the operator $A$. If $A$ is unbounded with domain $\mathcal{D}(A)$, it is common to assume regularity on $u_0$ (e.g., $u_0\in\mathcal{D}(A^{\nu})$ for some $\nu>0$) to obtain asymptotic rates of convergence. In particular, the important problem of explicit error control of solutions for arbitrary initial data is largely open. In this paper, we consider the following question:

\vspace{1mm}
\begin{itemize}[leftmargin=2mm,rightmargin=2mm]
	\item[]{\bf{Q.1:}} \textit{Can we compute semigroups with error control? That is, does there exist an algorithm that when given a generator $A$ of a strongly continuous semigroup on $\mathcal{H}$, time $t>0$, arbitrary $u_0\in \mathcal{H}$ and error tolerance $\epsilon>0$, computes an approximation of $\exp(tA)u_0$ to accuracy $\epsilon$ in $\mathcal{H}$?}
\end{itemize}
\vspace{1mm}

\noindent{}We provide a positive answer in \cref{disc_thm1}, with minimal assumptions on the operator $A$ and the initial condition $u_0$. Our method combines a regularized functional calculus, suitable contour quadrature rules, and the adaptive computation of resolvents in infinite dimensions. To the best of our knowledge, this provides the first answer to Q.1.

A prototypical example of \eqref{cauchy_prob} is when $A$ is a partial differential operator (PDO) on some domain. For unbounded domains, such as $\mathcal{H}=L^2(\mathbb{R}^d)$, this is a well-studied yet notoriously difficult challenge. The methods listed above yield invaluable insight into many computational issues. However,  the answer to Q.1 for unbounded domains remains largely unknown in the general case. For example, only in specific cases does one know how to truncate the domain and set appropriate boundary conditions. Even if one can prove the \textit{existence} of suitable truncations and boundary conditions, there may not be an \textit{algorithm} that does this (the original results of \cite{engquist1979radiation} reflect this). Moreover, difficulties are intensified in the case of irregular geometry or variable coefficients (see \cref{AB_example}). For the case of the Schr\"odinger equation,
\begin{equation}\setlength\abovedisplayskip{5pt}\setlength\belowdisplayskip{5pt}
\label{cauchy_prob_scrod}
i\frac{\partial u}{\partial t}=-\Delta u+Vu,\quad u_0\in L^2(\mathbb{R}^d),
\end{equation}
Q.1 has only just been answered for general classes of potential $V$ by using weighted Sobolev bounds on the initial condition for rigorous domain truncation \cite{becker2020computing}. In light of this, a second question we consider is the following:

\vspace{1mm}
\begin{itemize}[leftmargin=2mm,rightmargin=2mm]
	\item[]{\bf{Q.2:}} \textit{For $\mathcal{H} = L^2(\mathbb{R}^d)$,
	is there a large class of PDO generators $A$ (more general than \eqref{cauchy_prob_scrod}) on the unbounded domain $\mathbb{R}^d$ where the answer to Q.1 is yes?}
\end{itemize}
\vspace{1mm}

\noindent{}We provide a positive answer in \cref{PDO_thm1}, for PDOs formally defined by
$$\setlength\abovedisplayskip{5pt}\setlength\belowdisplayskip{5pt}
[{A}u](x)=\sum_{k\in\mathbb{Z}_{\geq0}^d,\|k\|_{\infty}\leq N} a_k(x)\partial^ku(x),
$$
with minimal regularity assumptions on the coefficients $a_k$. Our method uses Hermite functions (for convenience only) to reduce the problem to Q.1 via quasi-Monte Carlo numerical integration. Similar results can be shown via this technique for domains different to $\mathbb{R}^d$ and using other choices of basis.

The solution of \eqref{cauchy_prob} is, at least formally, the Bromwich complex contour integral
\begin{equation}\setlength\abovedisplayskip{5pt}\setlength\belowdisplayskip{5pt}
\label{bromwich_int}
\exp(tA)u_0=\left[\frac{-1}{2\pi i}\int^{\sigma+i\infty}_{\sigma-i\infty}e^{zt}(A-zI)^{-1}\,dz\right]u_0, \quad \text{for sufficiently large $\sigma\in\mathbb{R}$,}
\end{equation}
and computing solutions of \eqref{cauchy_prob} is a special case of inverting an operator-valued Laplace transform. The first use of \eqref{bromwich_int} as a method for solving the heat equation goes back to Talbot \cite{talbot1979accurate}, though with no reported numerical results for time-evolution problems. For early numerical work on this problem, see Gavrilyuk and Makarov \cite{gavrilyuk2001exponentially}, as well as Sheen, Sloan, and Thom{\'e}e \cite{sheen2003parallel}. Since these early works, there have been numerous methods using \eqref{bromwich_int}, with a focus on parabolic PDEs \cite{weideman2010improved,weideman2007parabolic,lopez2006spectral,dingfelder2015improved,mclean2004time,sheen2003parallel,gavrilyuk2011exponentially} (see also the discussion in \cref{analytic_semigp_sec}). For analytic semigroups, one can deform the contour of integration to obtain exponential decay of the integrand (see \cref{fig:cont_examples}).

An excellent survey of contour methods is provided in the paper \cite{trefethen2014exponentially} of Trefethen and Weideman. Contour methods possess many potential advantages, particularly when the resolvents $(A-zI)^{-1}$ can be computed efficiently. The linear systems that result from quadrature rules can be solved in parallel, and their solutions can be reused for different times. When direct methods are impractical, Krylov subspace methods are useful since only one Krylov basis needs to be constructed and computations for different $z$ reduce to a sequence of upper-Hessenberg systems of small dimension \cite{simoncini2007recent}. In the present paper, we demonstrate an additional key advantage: contour methods can be used in the \textit{infinite-dimensional} setting to tackle Q.1 directly, even for non-analytic semigroups. In the general case, contour deformations may not be possible and one needs to be careful even when defining $\exp(tA)$ for unbounded, possibly non-normal operators $A$ \cite{batty2009unbounded}. For example, the integrand in \eqref{bromwich_int} is not absolutely convergent. To overcome this issue, we combine a regularized version of the functional calculus and numerical quadrature of an appropriate contour integral. We compute the resolvent in an adaptive manner, providing explicit error control.

Dealing with the operator $A$ directly, as opposed to a truncation or discretization, allows us to provide rigorous convergence results under quite general assumptions. In many problems, there is an additional practical benefit in that it is easier to bound the resolvent (see \eqref{Nrange_resbound} for how to do this using the numerical range). In contrast, previous approaches to \eqref{cauchy_prob} are typically of the flavor ``truncate-then-solve.'' A truncation/discretization of $A$ is adopted and methods for computing the exponential of a finite matrix are used. In rigorously answering Q.1, it is vital to adopt a ``solve-then-discretize'' approach\footnote{The term ``solve-then-discretize'' can be traced at least as far back as the \texttt{Chebfun} package for computing with functions in MATLAB \cite{driscoll2014chebfun}.}, with the main steps outlined in \cref{alg:l2_semigroup}. The ``solve-then-discretize'' paradigm has recently been applied to spectral computations \cite{colbrook2019compute,horning2020feast,colbrook2019foundations,colbrook2019computationGEOM,johnstone2021bulk}, extensions of classical methods such as the QL and QR algorithms \cite{webb_thesis,colbrook2019infinite} (see also \cite{townsend2015continuous}), Krylov methods \cite{olver2009gmres,gilles2019continuous}, and spectral measures \cite{webb2017spectra,colbrook2019computing,colbrook2020computing}. Related work includes that of Olver, Townsend and Webb, providing a foundational and practical framework for infinite-dimensional numerical linear algebra and computations with infinite data structures \cite{Olver_Townsend_Proceedings, Olver_SIAM_Rev, Olver_code1, Olver_code2}.

\subsection{Summary of main results} Paraphrases of our main theorems are:

\vspace{0.5mm}

\noindent{}\textbf{Theorem \ref{disc_thm1}:}\! Given an infinite matrix representation of $A$, the answer to Q.1 is yes.

\noindent{}\textbf{Theorem \ref{PDO_thm1}:}\! The answer to Q.2 is yes for PDOs that generate a strongly continuous semigroup, and whose coefficients are polynomially bounded and of locally bounded total variation. This result holds even if we only allow our algorithm to point sample the coefficients and can be extended to domains other than $\mathbb{R}^d$.

\vspace{0.5mm}

\noindent{}\textbf{Theorem \ref{an_quad_theorem}:}\! We provide a stable and rapidly convergent quadrature rule for analytic semigroups, summarized in \cref{alg:spec_meas}. The quadrature error decreases like $\exp(-cN/\log(N))$ for $N$ quadrature points and the quadrature remains stable as $N\rightarrow\infty$.\footnote{This is a well-known difficulty in the literature - instabilities occur for quadrature rules whose points have unbounded real part as $N\rightarrow\infty$ \cite{weideman2010improved}. An advantage of a stable quadrature rule is that it no longer becomes essential to determine an optimal value of $N$ or optimal contour parameters, which often requires a heavy burden of case-specific analysis on the user. Taking $N$ larger no longer incurs a stability penalty, and fine-tuning to achieve high accuracy is no longer needed.} To deal with infinite-dimensional operators, contours are evaluated far from the spectrum of $A$ so that computing $(A-zI)^{-1}$ does not become prohibitively expensive or require large truncations. Our quadrature rule can also be used for inverting more general Laplace transforms (\cref{complex_diff_example}).

\vspace{0.5mm}

We demonstrate the practicality and versatility of our approach on a range of examples in \cref{num_exams_sec}. The only implementation requirement is computing $(A-zI)^{-1}$ with error bounds. This flexibility allows users to compute semigroups for a wide range of problems and employ their favorite discretization.

\subsection{Extensions to high-order Cauchy problems}

Our results extend to
\begin{equation}\setlength\abovedisplayskip{6pt}\setlength\belowdisplayskip{6pt}
\label{ACPN}
u^{(N)}+A_{N-1}u^{(N-1)}+\cdots+A_0 u=0 \text{ for } t\geq 0,\quad u^{(j)}(0)=u_j \text{ for }j=0,...,N-1,
\end{equation}
for suitable operators $A_0,...,A_{N-1}$. Here, the notation $u^{(l)}$ means the $l$th derivative of $u$ with respect to time. One can study \eqref{ACPN} directly, under the assumption that $u_j\in\mathcal{H}$, via the characteristic polynomial $p_{(N)}(z)=z^N+\sum_{j=0}^{N-1}z^{j}A_k$ and the generalized resolvent $R_{(N)}(z)=p_{(N)}(z)^{-1}$. However, a stronger form of the usual well-posedness \cite{xiao2013cauchy} is required when $N>1$ to prevent paradoxical situations such as loss of exponential bounds of solutions and non-existence of phase spaces \cite{fattorini2011second}. A generalization of the Hille--Yosida theorem holds under these conditions \cite[Ch. 2]{xiao2013cauchy} and our methods of building algorithms with error control can be extended under strong well-posedness by replacing $(A-zI)^{-1}$ with $R_{(N)}(z)$. A special case, where strong well-posedness is not needed, is the generalized wave equation
$$\setlength\abovedisplayskip{6pt}\setlength\belowdisplayskip{6pt}
u''(t)=Au(t),t\geq 0\qquad u(0)=u_0\in\mathcal{H}, u'(0)=u_1\in\mathcal{H},
$$
whose solutions can be computed using sines and cosines of $\sqrt{-A}$.

A far more common approach reduces \eqref{ACPN} to the first-order system \cite[Ch. VI]{engel1999one}
$$\setlength\abovedisplayskip{6pt}\setlength\belowdisplayskip{6pt}
\mathcal{U}'=\mathcal{A}\mathcal{U}\text{ for }t\geq 0,\quad\mathcal{A}=\begin{pmatrix}
0 & I &  & \\
 & 0 & I &     \\
  && \ddots  & \ddots \\
-A_0 & -A_1 & \cdots & -A_N
\end{pmatrix},\quad \mathcal{U}=\begin{pmatrix}
u \\
u^{(1)}\\
\vdots \\
u^{(N-1)}
\end{pmatrix}.
$$
This approach allows a more flexible treatment of initial conditions (e.g., $u_j$ could lie in different subspaces of a Banach space). Our results apply directly to this case if $\mathcal{A}$ generates a strongly continuous semigroup (see the example in \cref{Elastic_example}).

\subsection{Notation and outline} We use $\langle \cdot,\cdot\rangle$ to denote the inner product on $\mathcal{H}$. The induced norm and corresponding operator norm are denoted by $\|\cdot\|$. The identity operator on a Banach space $X$ is denoted by $I$ and the algebra of bounded linear operators on $X$ is denoted by $\mathcal{L}(X)$. The spectrum and domain of a linear operator $A$ are denoted by $\mathrm{Sp}(A)$ and $\mathcal{D}(A)$, respectively. The resolvent operator $(A-zI)^{-1}$, defined on the resolvent set $\rho(A):=\mathbb{C}\backslash\mathrm{Sp}(A)$, is denoted by $R(z,A)$. When dealing with $l^2(\mathbb{N})$ (the space of square summable sequences), $P_n$ denotes the orthogonal projection onto the span of the first $n$ canonical basis vectors.

In \cref{sec_maths_prelimns} we recall basic properties of strongly continuous semigroups and suitable definitions of computational problems in infinite-dimensional spaces. Results for the canonical Hilbert space $l^2(\mathbb{N})$ are presented in \cref{disc_thm_sec} and extended to partial differential operators in \cref{PDE_section}. Quadrature rules for analytic semigroups are presented in \cref{analytic_semigp_sec}, along with numerical examples. Further numerical examples are given in \cref{num_exams_sec} and concluding remarks in \cref{conc_sect}. We also include an Appendix of some results needed in our proofs.

\section{Mathematical preliminaries}\label{sec_maths_prelimns}

\subsection{Recalling basic properties of semigroups}\label{FA_recall}

We first recall the following two definitions, found in any textbook that treats semigroup theory, for example \cite{pazy2012semigroups,arendt2001cauchy}. The first definition defines a semigroup, and the second provides standard notions of a solution of \eqref{cauchy_prob}. \cref{thm_neededn} then connects these two definitions.

\begin{definition}
\label{scsg}
A strongly continuous semigroup ($C_0$-semigroup) on a Banach space $X$ is a map $S:[0,\infty)\rightarrow \mathcal{L}(X)$ such that
\begin{enumerate}
	\item $S(0)=I$
	\item $S(s+t)=S(s)S(t),\quad\forall s,t\geq 0$
	\item $S(t)$ converges strongly to $I$ as $t\downarrow 0$ (i.e., $\lim_{t\downarrow0}S(t)x=x,\text{ for all } x\in X$).
\end{enumerate}
The infinitesimal generator $A$ of $S$ is defined as $Ax=\lim_{t\downarrow 0}\frac{1}{t}(S(t)-I)x$,
where $\mathcal{D}(A)$ is all $x\in X$ such that the limit exists, and we write $S(t)=\exp(tA)$.
\end{definition}

\begin{definition}
\label{cauchy}
A continuous function $u:[0,\infty)\rightarrow X$ is a
\begin{enumerate}
	\item Classical solution of the Cauchy problem \eqref{cauchy_prob} if it is continuously differentiable, $u(t)\in \mathcal{D}(A)$ for all $t\geq 0$, and \eqref{cauchy_prob} is satisfied,
	\item Mild solution of the Cauchy problem \eqref{cauchy_prob} if for all $t\geq 0$,
	$$
	\int_0^t u(s)ds\in\mathcal{D}(A)\quad \text{and}\quad A \int_0^t u(s)ds = u(t)-u_0.
	$$
\end{enumerate}
\end{definition}

\vspace{-1mm}The following theorem tells us precisely when a unique mild solution exists.

\begin{theorem}[Theorem 3.1.12 of \cite{arendt2001cauchy}]\label{thm_neededn} Let $A$ be a closed operator acting on the Banach space $X$. The following assertions are equivalent:
\begin{enumerate}[leftmargin=9mm]
	\item[(a)] For any $u_0\in X$, there exists a unique mild solution of \eqref{cauchy_prob}.
	\item[(b)] $\rho(A)\neq\emptyset$ and for every $u_0\in\mathcal{D}(A)$, there is a unique classical solution of \eqref{cauchy_prob}.
	\item[(c)] The operator $A$ generates a $C_0$-semigroup $S$.
\end{enumerate}
When these conditions hold, the solution is given by $u(t)=S(t)u_0=\exp(tA)u_0$.
\end{theorem}
The Hille--Yosida theorem tells us precisely when an operator $A$ generates a strongly continuous semigroup, and thus, by \cref{thm_neededn}, when \eqref{cauchy_prob} admits a unique solution.
\begin{theorem}[Hille--Yosida theorem]
\label{HY_th}
A closed operator $A$ on $X$ generates a $C_0$-semigroup if and only if $A$ is densely defined and there exists $\omega\in\mathbb{R}$, $M>0$ with
\begin{itemize}
	\item[(1)] $\{\lambda\in\mathbb{R}:\lambda>\omega\}\subset\rho(A)$.
	\item[(2)] For all $\lambda>\omega$ and $n\in\mathbb{N}$, $(\lambda-\omega)^n\|R(\lambda,A)^n\|\leq M.$
\end{itemize}
Under these conditions, $\|\exp(tA)\|\leq M\exp(\omega t)$ and if $\mathrm{Re}(\lambda)\!>\!\omega$ then $\lambda\!\in\!\rho(A)$ with
\begin{equation}
	\label{Hille_YBD}
	\|R(\lambda,A)^n\|\leq\frac{M}{(\mathrm{Re}(\lambda)-\omega)^n},\quad \text{for all } n\in\mathbb{N}.
	\end{equation}
\end{theorem}

We show below that \cref{HY_th} can be exploited to give an algorithm that positively answers Q.1. In particular, the resolvent bound in \eqref{Hille_YBD} is used in the proof of \cref{disc_thm1} and allows error estimates for regularized contour integrals. 

\subsection{Computational problems}

Since the general Hilbert space $\mathcal{H}$ we consider is infinite-dimensional, care must be taken when defining a computational problem and when stating the information that we allow our algorithms to access. We begin with a precise and general definition of a computational problem, following the setup of Solvability Complexity Index (SCI) hierarchy \cite{antun2021can,hansen2011solvability,colbrook2020foundations,ben2015can}. The SCI hierarchy provides a general framework for scientific computation.

\begin{definition}
\label{def_comp_problem}
A collection $\{\Xi,\Omega,\mathcal{M},\Lambda\}$ is a computational problem if:
\begin{itemize}[noitemsep,leftmargin=5mm]
\item[(i)] \underline{Domain}: $\Omega$ is some set.
\item[(ii)] \underline{Evaluation set} that distinguishes elements of $\Omega$: $\Lambda$ is a set of complex-valued functions on $\Omega$, such that if $\iota_1,\iota_2\in\Omega$ has $f(\iota_1)=f(\iota_2)$ for all $f\in\Lambda$, then $\iota_1=\iota_2$.
\item[(iii)] \underline{Problem function}: $\Xi:\Omega\to \mathcal{M}$, where $\mathcal{M}$ is a metric space with metric $d_{\mathcal{M}}$.
\end{itemize}
\end{definition}

The domain $\Omega$ is the set of objects that give rise to our computational problems. The problem function $\Xi : \Omega\to \mathcal{M}$ describes what we want to compute. Finally, the evaluation set $\Lambda$ is the collection of functions that provide the information we allow algorithms to read. With this in hand, we can define a general algorithm, which shows the interplay and purpose of each part of a computational problem.
\begin{definition}\label{Gen_alg}
Given a computational problem $\{\Xi,\Omega,\mathcal{M},\Lambda\}$, a {general algorithm} is a mapping $\Gamma:\Omega\to \mathcal{M}$ such that for each $\iota\in\Omega$
\begin{itemize}[leftmargin=5mm]
\item[(i)] There exists a finite (non-empty) subset of evaluations $\Lambda_\Gamma(\iota) \subset\Lambda$, 
\item[(ii)] The action of $\,\Gamma$ on $\iota$ only depends on $\{\iota_f\}_{f \in \Lambda_\Gamma(\iota)}$ where $\iota_f := f(\iota),$
\item[(iii)] For every $\eta\in\Omega$ such that $\eta_f=\iota_f$ for every $f\in\Lambda_\Gamma(\iota)$, it holds that $\Lambda_\Gamma(\eta)=\Lambda_\Gamma(\iota)$.
\end{itemize}
We will sometimes write $\Gamma(\{A_f\}_{f \in \Lambda_\Gamma(A)})$, to emphasize that $\Gamma(A)$ only depends on the results $\{A_f\}_{f \in \Lambda_\Gamma(A)}$ of finitely many evaluations. 
\end{definition}

These three properties are the most fundamental properties we would expect any deterministic computational device to obey. The first condition says that the algorithm can only take a finite amount of information, though it is allowed to adaptively choose the information, depending on the input it reads. The second condition ensures that the algorithm's output only depends on its input. The final condition ensures that the algorithm produces outputs consistently. The goal is for the algorithm $\Gamma$ to approximate the problem function $\Xi : \Omega\to \mathcal{M}$ in a suitable sense.

\vspace{1mm}

\textbf{The type of algorithms in this paper:} In this paper, we exclusively consider \textit{arithmetic algorithms} (shortened to ``algorithms''), meaning that $\Gamma$ is recursive in its input ($\{f(\iota)\}_{f \in \Lambda}$ for $\iota\in\Omega$) and outputs a finite string of complex numbers that can be identified with an element in $\mathcal{M}$. For example, when considering computations in $l^2(\mathbb{N})$, our algorithms compute a vector in $l^2(\mathbb{N})$ of finite support with respect to the canonical basis. By recursive we mean the following. If $f(\iota)\!  \in \! \mathbb{Q}\! +\! i\mathbb{Q}$ for all $f \in \Lambda$, $\iota \in \Omega$, then $\Gamma(\{f(\iota)\}_{f \in \Lambda})$ can be executed by a Turing machine. If $f(\iota) \in \mathbb{C}$ for all $f \in \Lambda$, then $\Gamma(\{f(\iota)\}_{f \in \Lambda})$ can be executed by a Blum--Shub--Smale (BSS) machine \cite{Smale_book}. (In both cases with an oracle consisting of $\{f(\iota)\}_{f \in \Lambda}$). \textit{The reader need not worry about these matters}, but note that this means that our algorithms can be adapted and executed rigorously through methods such as interval arithmetic \cite{tucker2011validated,rump2010verification}.\footnote{The results of the current paper can be interpreted in terms of the SCI hierarchy. A computational problem $\{\Xi,\Omega,\mathcal{M},\Lambda\}$ lies in $\Delta_1^A$ if there exists an algorithm $\Gamma$ such that $d_{\mathcal{M}}(\Gamma(\iota,\epsilon),\Xi(\iota))\leq \epsilon$ for all $\iota\in\Omega$ and $\epsilon>0.$ For the other classes in the SCI hierarchy, see \cite{colbrook2020foundations}. For example, most infinite-dimensional spectral problems of interest do not lie in $\Delta_1^A$ and there is a classification theory determining which spectral problems can be solved, and with what type of algorithm \cite{colbrook2020foundations}.} For all of the numerical examples of this paper, we have performed computations using standard IEEE double-precision floating-point arithmetic.

\section{$C_0$-semigroups on $l^2(\mathbb{N})$ can be computed with error control}\label{disc_thm_sec}
First, we consider the canonical separable Hilbert space $l^2(\mathbb{N})$ of square summable sequences, using $e_1$, $e_2$, $\hdots$ to denote the canonical orthonormal basis. Let $\mathcal{C}(l^2(\mathbb{N}))$ denote the set of closed and densely defined linear operators $A$ such that $\mathrm{span}\{e_n:n\in\mathbb{N}\}$ forms a core of $A$ and its adjoint $A^*$ \cite[Ch. 3]{kato2013perturbation}. If $A\in\mathcal{C}(l^2(\mathbb{N}))$, then we can associate an infinite matrix with the operator $A$ through the inner products $A_{j,k}=\langle Ae_k, e_j \rangle$. Given $(A,u_0)\in\mathcal{C}(l^2(\mathbb{N}))\times l^2(\mathbb{N})$, we consider the following evaluation functions (recall that this is the readable input to our algorithm), denoted by $\Lambda_1$, which include the case of inexact input:
\begin{itemize}[noitemsep,leftmargin=4mm]
	\item \underline{Matrix evaluation functions:} $\{f_{j,k,m}^{(1)},f_{j,k,m}^{(2)}:j,k,m\in\mathbb{N}\}$ such that
$$
|f_{j,k,m}^{(1)}(A)-\langle Ae_k,e_j\rangle|\leq 2^{-m},\quad |f_{j,k,m}^{(2)}(A)-\langle Ae_k,Ae_j\rangle|\leq 2^{-m}, \quad \forall j,k,m\in\mathbb{N}.
$$
\item \vspace{-0.5mm}\underline{Coefficient and norm evaluation functions:} $\{f_{j,m}:j\in\mathbb{N}\cup\{0\},m\in\mathbb{N}\}$ such that
\begin{equation}\label{needinlemma}
|f_{0,m}(u_0)-\langle u_0,u_0\rangle|\leq 2^{-m},\quad|f_{j,m}(u_0)-\langle u_0,e_j\rangle|\leq 2^{-m},\quad\forall j,m\in\mathbb{N}.
\end{equation}
\end{itemize}

Following \cref{def_comp_problem}, we let $\Omega_{C_0}$ denote the set of triples $(A,u_0,t)$ where $A\in \mathcal{C}(l^2(\mathbb{N}))$ generates a strongly continuous semigroup, $u_0\in l^2(\mathbb{N})$ and $t>0$. We define the set of evaluation functions for such triples to be $\Lambda_{C_0}=\Lambda_1\cup\{M(A),\omega(A)\}$, where $M=M(A)$ and $\omega=\omega(A)$ are constants satisfying the conditions in \cref{HY_th} for the generator $A$. Finally, we consider the problem function $\Xi_{C_0}:\Omega_{C_0}\rightarrow l^2(\mathbb{N}), (A,u_0,t)\mapsto \exp(tA)u_0$. In other words, the computation of the solution of \eqref{cauchy_prob}. The following theorem provides a positive answer to Q.1 in the introduction.

\begin{theorem}[$C_0$-semigroups on $l^2(\mathbb{N})$ computed with error control]\label{disc_thm1}
There exists an algorithm $\Gamma$ using $\Lambda_{C_0}$ such that for any $\epsilon>0$ and $(A,u_0,t)\in\Omega_{C_0}$,
$$\setlength\abovedisplayskip{5pt}\setlength\belowdisplayskip{5pt}
\|\Gamma(A,u_0,t,\epsilon)-\exp(tA)u_0\|\leq \epsilon.
$$
\end{theorem}

The algorithm is summarized in \cref{alg:l2_semigroup}, which outlines the key steps. First, we write the exponential $\exp(tA)$ as an absolutely convergent integral using a regularized functional calculus. We then reduce the problem to computing this integral acting on $\mathrm{span}\{e_n:n\in\mathbb{N}\}$. Finally, the integral is computed using quadrature and adaptive computation of the resolvent with error control.

\begin{algorithm}[t]
\textbf{Input:} $(A,u_0,t)\in\Omega_{C_0}$ through evaluation set $\Lambda_{C_0}$, where $A$ generates a strongly continuous semigroup, $u_0\in l^2(\mathbb{N})$ and $t>0$, as well as error bound $\epsilon>0$. \\
\vspace{-2mm}
$$
\text{Let }B=\left[\frac{1}{2\pi i}\int^{\omega+1+i\infty}_{\omega+1-i\infty}\frac{e^{zt}R(z,A)}{(z-(\omega+2))^2}\,dz\right]\text{ denote the regularized integral.}
$$
\begin{algorithmic}[1]
\STATE Compute $u_0^\epsilon=\sum_{j=1}^{J_0}u_{0j}^{(\epsilon)}e_j,$ such that $\|u_0-u_0^\epsilon\|\leq \epsilon \exp(-\omega t)/(2M)$.
\STATE Compute $u_1^\epsilon=\sum_{j=1}^{J_1}u_{1j}^{(\epsilon)}e_j$ such that $\|u_1^{\epsilon}-(A-(\omega+2)I)u_0^\epsilon\|\leq \epsilon \exp(-\omega t)/(2M^2)$.
\STATE Compute $u_2^\epsilon=\sum_{j=1}^{J_2}u_{2j}^{(\epsilon)}e_j$ such that $\|u_2^{\epsilon}+(A-(\omega+2)I)u_1^\epsilon\|\leq \epsilon \exp(-\omega t)/(2M^2)$.
\STATE Compute weights $w_j$ and nodes $z_j$ such that $\|B-\sum_{j=1}^Nw_j R(z_j,A)\|\leq\epsilon/16$.
\STATE Using the input bounds on the resolvent provided by the Hille--Yosida theorem and \cref{res_est1} to control the total error in computing the resolvents, compute $r_j^{\epsilon}\approx R(z_j,A)u_2^\epsilon$ so that $\sum_{j=1}^N|w_j|\|r_j^\epsilon-R(z_j,A)u_2^\epsilon\|\leq \epsilon/16.$
\end{algorithmic} \textbf{Output:} $\Gamma(A,u_0,t,\epsilon)=\sum_{j=1}^Nw_jr_j^\epsilon$, an $\epsilon$-accurate approximation of $\exp(tA)u_0$.
\caption{Algorithm for computing semigroups on $l^2(\mathbb{N})$ with error control. The details of how to do each step are provided in the proof of \cref{disc_thm1}. Steps 2, 3 and 5 require a solve-then-discretize approach for computing the action of $A$ and the resolvent $R(z,A)$ with error control via the available input information $\Lambda_{C_0}$. We have used the notation $\omega=\omega(A)$ and $M=M(A)$ (constants in \cref{HY_th}).}\label{alg:l2_semigroup}
\end{algorithm}

\begin{proof}[Proof of \cref{disc_thm1}]
Let $(A,u_0,t)\in\Omega_{C_0}$, and set $\omega=\omega(A)$ and $M=M(A)$ throughout the notation of this proof. It suffices to show that given any $\epsilon>0$, we can compute a $\epsilon$-accurate approximation of $\exp(tA)u_0$. We use the following \textit{regularization} to define the needed holomorphic functional calculus \cite{batty2013holomorphic}:
$$\setlength\abovedisplayskip{5pt}\setlength\belowdisplayskip{5pt}
\exp(tA)u_0=-(A-(\omega+2)I)^2\underbrace{\left[\frac{1}{2\pi i}\int^{\omega+1+i\infty}_{\omega+1-i\infty}\frac{e^{zt}R(z,A)}{(z-(\omega+2))^2}\,dz\right]}_{:=B}u_0,
$$

\vspace{-0.5mm}\noindent{}where the integral is taken in the direction $\omega+1-i\infty$ to $\omega+1+i\infty$. The point of this regularization is that the integral now converges absolutely. With this representation in hand, we can now prove \cref{disc_thm1} via the following three steps.

\textbf{Step 1: Reduction to a finite sum.} Using \cref{boring} in the appendix, we can compute $u_0^\epsilon=\sum_{j=1}^{J_0}u_{0j}^{(\epsilon)}e_j,$ such that $\|u_0-u_0^\epsilon\|\leq \epsilon \exp(-\omega t)/(2M)$ (approximating this quantity from below using $\Lambda_{C_0}$). This bound is chosen so that, using $\|\exp(tA)\|\leq M\exp(\omega t)$, we have $\|\exp(tA)(u_0-u_0^\epsilon)\|\leq\epsilon/2$. So it suffices to compute $\exp(tA)u_0^\epsilon$ to accuracy $\epsilon/2$.

\textbf{Step 2: Reduction to an absolutely convergent integral.} Since $u_0^\epsilon$ is a finite sum of elements in $\mathcal{D}(A)$, $u_0^\epsilon\in\mathcal{D}(A)$ and we can rewrite the exponential as
$$\setlength\abovedisplayskip{5pt}\setlength\belowdisplayskip{5pt}
\exp(tA)u_0^\epsilon=-(A-(\omega+2)I)B(A-(\omega+2)I)u_0^\epsilon.
$$
Note that the evaluation functions in $\Lambda_{C_0}$ allow us to compute the inner products
$
\langle(A-(\omega+2)I)u_0^\epsilon,(A-(\omega+2)I)u_0^\epsilon\rangle\text{ and } \langle(A-(\omega+2)I)u_0^\epsilon,e_j\rangle$ to any specified accuracy. We can therefore apply \cref{boring} to the vector $(A-(\omega+2)I)u_0^\epsilon$ to compute $u_1^\epsilon=\sum_{j=1}^{J_1}u_{1j}^{(\epsilon)}e_j$ such that $\|u_1^{\epsilon}-(A-(\omega+2)I)u_0^\epsilon\|\leq \epsilon \exp(-\omega t)/(2M^2)$. Since
\begin{equation}\setlength\abovedisplayskip{5pt}\setlength\belowdisplayskip{5pt}
\label{split1}
\left\|{(A\!-\!(\omega+2)I)}B\right\|\!\leq \!\|\exp(tA)\|\|R(\omega+2,A)\|\!\leq\! {M^2e^{\omega t}}/{2},
\end{equation}
we must have that $\|\exp(tA)u_0^\epsilon+(A-(\omega+2)I)Bu_1^\epsilon\|\leq \epsilon/{4}.$ So it suffices to compute $-(A-(\omega+2)I)Bu_1^\epsilon$ to accuracy $\epsilon/4$. Since $u_1^\epsilon\in\mathcal{D}(A)$ (it is a finite sum of elements in $\mathcal{D}(A)$), $-(A-(\omega+2)I)Bu_1^\epsilon=-B(A-(\omega+2)I)u_1^\epsilon$. Using the same argument as before, we can compute $u_2^\epsilon=\sum_{j=1}^{J_2}u_{2j}^{(\epsilon)}e_j$ such that $\|u_2^\epsilon+(A-(\omega+2)I)u_1^\epsilon\|\leq\epsilon \exp(-\omega t)/(2M^2)$. By \cref{HY_th},
$$\setlength\abovedisplayskip{6pt}\setlength\belowdisplayskip{6pt}
\left\|B\right\|\leq \|\exp(tA)\|\|R(\omega+2,A)^2\|\leq {M^2\exp(\omega t)}/{4}.
$$
It follows that $\|Bu_2^\epsilon+(A-(\omega+2)I)Bu_1^\epsilon\leq\|\leq \epsilon/8$. Combining these inequalities, we therefore have that
$$\setlength\abovedisplayskip{6pt}\setlength\belowdisplayskip{6pt}
\|\exp(tA)u_0-Bu_2^\epsilon\|\leq \frac{\epsilon}{2}+\frac{\epsilon}{4}+\frac{\epsilon}{8}.
$$
Hence it suffices to compute an $\epsilon/8$-approximation of $Bu_2^\epsilon$. Since $u_2^\epsilon$ is a finite linear combination of canonical basis vectors, it suffices to show that for any $l\in\mathbb{N}$ we can compute the following to any given accuracy:
\begin{equation}\setlength\abovedisplayskip{5pt}\setlength\belowdisplayskip{5pt}
\label{cauchy_c0D}
B e_l=\left[\frac{1}{2\pi i}\int^{\omega+1+i\infty}_{\omega+1-i\infty}\frac{e^{zt}R(z,A)}{(z-(\omega+2))^2}\,dz\right]e_l.
\end{equation}

\textbf{Step 3: Approximation of the integral through adaptive approximation of the resolvent and quadrature.} To prove this, fix $l\in\mathbb{N}$ and define
\begin{equation}\setlength\abovedisplayskip{6pt}\setlength\belowdisplayskip{6pt}
\label{F_def}
F(s)=\frac{\exp((\omega+1)t+ist)}{2\pi(is-1)^2}R(\omega+1+is,A)e_l\in l^2(\mathbb{N}).
\end{equation}
Then \eqref{cauchy_c0D} can be re-written as $\int_{\mathbb{R}}F(s)ds.$ Note that by \eqref{Hille_YBD},
\begin{equation}\setlength\abovedisplayskip{6pt}\setlength\belowdisplayskip{6pt}
\label{F_bdd2}
\|F(s)\|\leq {M\exp((1+\omega)t)}{(1+s^2)^{-1}}/(2\pi).
\end{equation}
Given $\hat\epsilon>0$, we can therefore compute a cut-off $L\in\mathbb{N}$ such that
$$\setlength\abovedisplayskip{6pt}\setlength\belowdisplayskip{6pt}
\Big\|\int_{|y|>L}F(s)\,ds\Big\|\leq \frac{M\exp((1+\omega)t)}{\pi}\int_{s>L}{(1+s^2)}^{-1}\,ds\leq \frac{M\exp((1+\omega)t)}{\pi L}\leq\hat\epsilon.
$$
Hence it suffices to compute $\int_{-L}^L F(s)\,ds$ to arbitrary accuracy. By bounding each term of the derivative separately and using \eqref{Hille_YBD}, we have
\begin{equation}\setlength\abovedisplayskip{6pt}\setlength\belowdisplayskip{6pt}
\label{Fd_bdd2}
\|F'(s)\|\leq (3+t)\|F(s)\|\leq{(3+t)M\exp((1+\omega)t)}/({2\pi}).
\end{equation}
Using \eqref{F_bdd2} and \eqref{Fd_bdd2}, for a given $\hat\epsilon>0$, we can compute an integer $m$ such that
$$\setlength\abovedisplayskip{6pt}\setlength\belowdisplayskip{6pt}
\Big\|\int_{-L}^L F(s)\,ds-\frac{1}{m}\sum_{j=-L\cdot m+1}^{L\cdot m}F\left(\frac{j}{m}\right)\Big\|\leq \frac{2L\|F'\|_{\infty}}{m} \leq \hat\epsilon.
$$
Hence it suffices to be able to compute $F(q)$ for rational $q$ to arbitrary accuracy. This follows from \cref{res_est1} (applied to the shifted operator $T=A-(\omega+1+iq)I$), standard approximations of ${\exp((\omega+1)t+iqt)}/({2\pi(iq-1)^2})$ and the fact that we can use \eqref{Hille_YBD} to bound the resolvent norm appearing in \cref{res_est1}.
\end{proof}

In practice, one uses a method such as Gaussian quadrature for the truncated integral in the final step. The resulting error bounds after truncation decrease exponentially in the number of quadrature points. Optimal contours could also be found by studying regions of analyticity and bounds on the function $F$ in \eqref{F_def}. For large times, one can also use quadrature methods for oscillatory integrals \cite{deano2017computing}. Such bounds are more complicated, so we present the above contour and quadrature rule in the proof. We provide a detailed analysis of quadrature for analytic semigroups in \cref{analytic_semigp_sec}.

\section{Extension to partial differential operators on $L^2(\mathbb{R}^d)$}\label{PDE_section} We now extend the above technique to PDOs. As an example, we consider the closure, denoted by $A$, of the initial operator
\begin{equation}\setlength\abovedisplayskip{5pt}\setlength\belowdisplayskip{6pt}
\label{PDO_init}
[\tilde{A}u](x)=\!\!\!\!\sum_{k\in\mathbb{Z}_{\geq0}^d,\left|k\right|\leq N}\!\!\!\!\! a_k(x)\partial^ku(x),\quad \mathcal{D}(\tilde{A})=\{u\text{ smooth with compact support}\}.
\end{equation}

\vspace{-2mm}\noindent{}We use multi-index notation with $\left|k\right|=\max\{\left|k_1\right|,...,\left|k_d\right|\}$ and $\partial^k=\partial^{k_1}_{x_1}\partial^{k_2}_{x_2}...\partial^{k_d}_{x_d}$. We assume that $\tilde{A}$ is closable and that the coefficients $a_k(x)$ are complex-valued measurable functions on $\mathbb{R}^d$. For dimension $d$ and $r>0$, consider the space
$$\setlength\abovedisplayskip{6pt}\setlength\belowdisplayskip{6pt}
\mathcal{A}_r=\{f\in \mathrm{Meas}([-r,r]^d):\left\|f\right\|_{\infty}+\mathrm{TV}_{[-r,r]^d}(f)<\infty\},
$$

\vspace{-1mm}\noindent{}where $\mathrm{Meas}([-r,r]^d)$ denotes the set of measurable functions on the hypercube $[-r,r]^d$ and $\mathrm{TV}_{[-r,r]^d}$ the total variation norm in the sense of Hardy and Krause \cite{niederreiter1992random}. This space becomes a Banach algebra when equipped with the norm \cite{blumlinger1989topological}
$$\setlength\abovedisplayskip{6pt}\setlength\belowdisplayskip{6pt}
\left\|f\right\|_{\mathcal{A}_r}:=\left\|f|_{[-r,r]^d}\right\|_{\infty}+(3^d+1)\mathrm{TV}_{[-r,r]^d}(f).
$$

\vspace{-1mm}\noindent{}We let $\Omega_{\mathrm{PDE}}$ be all such $(A,u_0,t)$ with $u_0\in L^2(\mathbb{R}^d)$ and $t>0$, for which $A$ generates a strongly continuous semigroup on $L^2(\mathbb{R}^d)$ and the following hold:

\vspace{1mm}

\begin{enumerate}[leftmargin=10mm]
	\item[(1)] The set of smooth, compactly supported functions forms a core of $A$ and $A^*$.
	\item[(2)] \underline{At most polynomial growth:} There exist positive constants $C_k$ and integers $B_k$ such that almost everywhere on $\mathbb{R}^d$, $|a_k(x)|\leq C_k(1+|x|^{2B_k}).$
	\item[(3)] \underline{Locally bounded total variation:} For all $r>0$, $u_0|_{[-r,r]^d},a_k|_{[-r,r]^d}\in\mathcal{A}_r$.
\end{enumerate}

\vspace{1mm}
These assumptions are very mild as the class of functions with locally bounded variation includes discontinuous functions and functions with arbitrary wild oscillations at infinity. For input $(A,u_0,t)\in\Omega_{\mathrm{PDE}}$, we define $\Lambda_{\mathrm{PDE}}$ as the set of evaluation functions (where ranges of indices have been suppressed for notational convenience):

\vspace{1mm}

\begin{itemize}[leftmargin=10mm]
	\item[(a)] \underline{Pointwise coefficient evaluations:} $\left\{S_{k,q,m}\right\}$ such that for all $m\in\mathbb{N}$,
$$\setlength\abovedisplayskip{4pt}\setlength\belowdisplayskip{4pt}
\left|S_{k,q,m}(A)-a_k(q)\right|\leq 2^{-m},\quad \forall q\in\mathbb{Q}^d.
$$
\item[(b)] \underline{Pointwise initial condition evaluations:} $\left\{S_{q,m}\right\}$ such that for all $m\in\mathbb{N}$,
$$\setlength\abovedisplayskip{4pt}\setlength\belowdisplayskip{4pt}
\left|S_{q,m}(u_0)-u_0(q)\right|\leq 2^{-m},\quad \forall q\in\mathbb{Q}^d.
$$
\item[(c)] \underline{Bounds on growth and total variation:} $\left\{C_k,B_k\right\}$ such that the bound in (2) holds and positive sequences $\left\{b_n\right\}_{n\in\mathbb{N}}$ and $\left\{c_n\right\}_{n\in\mathbb{N}}$ such that for all $n\in\mathbb{N}$,
$$\setlength\abovedisplayskip{4pt}\setlength\belowdisplayskip{4pt}
\max_{\left|k\right|\leq N}{\left\|a_k\right\|_{\mathcal{A}_n}}\leq{b_n},\quad {\left\|u_0\right\|_{\mathcal{A}_n}}\leq{c_n}.
$$
\item[(d)] \underline{Decay of initial condition:} A positive sequence $\left\{d_n\right\}_{n\in\mathbb{N}}$, such that
$$\setlength\abovedisplayskip{4pt}\setlength\belowdisplayskip{4pt}
\|u_0|_{[-n,n]^d}-u_0\|_{L^2(\mathbb{R}^d)}\leq d_n,\quad \lim_{n\rightarrow\infty}d_n=0,
$$
\end{itemize}
together with constants $M=M(A)>0$ and $\omega=\omega(A)>0$ satisfying the conditions in \cref{HY_th} for the generator $A$. We consider the problem function $\Xi_{\mathrm{PDE}}:\Omega_{\mathrm{PDE}}\rightarrow L^2(\mathbb{R}^d),$ $(A,u_0,t)\mapsto \exp(tA)u_0$. In other words, the computation of the solution of \eqref{cauchy_prob} for PDOs $A$ on $L^2(\mathbb{R}^d)$. The following theorem provides a positive answer to Q.2 in the introduction.

\begin{theorem}[PDO $C_0$-semigroups on $L^2(\mathbb{R}^d)$ computed with error control]\label{PDO_thm1}
There exists an algorithm $\Gamma$ using $\Lambda_{\mathrm{PDE}}$ such that for any $\epsilon>0$ and $(A,u_0,t)\in\Omega_{\mathrm{PDE}}$,
$$\setlength\abovedisplayskip{4pt}
\|\Gamma(A,u_0,t,\epsilon)-\exp(tA)u_0\|\leq \epsilon.
$$
\end{theorem}

The algorithm is summarized in \cref{alg:PDO_semigroup}. \cref{PDO_thm1} says the following for PDOs with coefficients that are polynomially bounded and have locally bounded total variation. If the PDO generates a strongly continuous semigroup, then we can compute the semigroup with error control via point sampling the coefficients and initial condition. For example, we can solve the time-independent Schr\"odinger equation \eqref{cauchy_prob_scrod} on $L^2(\mathbb{R}^d)$ with error control (here $M(A)=1,w(A)=0$), for potentials satisfying the conditions in (2) and (3). As discussed in \cref{sec_intro_sec}, this is a decidedly difficult problem. Moreover, the result in \cref{PDO_thm1} is much more general than just Schr\"odinger operators. 

\begin{algorithm}[t]
\textbf{Input:} $(A,u_0,t)\in\Omega_{\mathrm{PDE}}$ through evaluation set $\Lambda_{\mathrm{PDE}}$, where PDO $A$ generates a strongly continuous semigroup, $u_0\in L^2(\mathbb{R}^d)$ and $t>0$, as well as error bound $\epsilon>0$.
\vspace{-4mm}
\begin{algorithmic}[1]
\STATE Convert the problem to one on $l^2(\mathbb{N})$ via
\begin{align*}
&\langle \hat Ae_k,\hat Ae_j \rangle = \int_{\mathbb{R}^d} (A\psi_{\pmb{m}(k)})\overline{(A\psi_{\pmb{m}(j)})}\,dx,\quad \langle \hat Ae_k,e_j \rangle = \int_{\mathbb{R}^d} (A\psi_{\pmb{m}(k)})\psi_{\pmb{m}(j)}\,dx \\
&\langle \hat u_0,e_j \rangle = \int_{\mathbb{R}^d} u_0\psi_{\pmb{m}(j)}\,dx, \quad \langle \hat u_0,\hat u_0 \rangle = \int_{\mathbb{R}^d} u_0\overline{u_0}\,dx,
\end{align*}
where $\{\psi_{m(j)}\}_{j\in\mathbb{N}}$ denotes an ordering of a tensor product Hermite basis of $L^2(\mathbb{R}^d)$. The integrals are computed using the total variation and growth bounds of coefficients/intial condition and quasi-Monte Carlo numerical integration.
\STATE Using step 1 to provide the $\Lambda_{C_0}$ evaluation functions, apply \cref{alg:l2_semigroup} with input $(\hat A, \hat u_0,t)$ and $\epsilon$. Call this output $\hat\Gamma(\hat A,\hat u_0,t,\epsilon)$.
\STATE Using the finite number of coefficients of $\hat\Gamma(\hat A,\hat u_0,t,\epsilon)$ computed in step 2, generate the solution
$$
\Gamma(A,u_0,t)=\sum_{j}[\hat\Gamma(\hat A,\hat u_0,t,\epsilon)]_j\psi_{\pmb{m}(j)}\in L^2(\mathbb{R}^d).
$$
\end{algorithmic} \textbf{Output:} $\Gamma(A,u_0,t)$, an $\epsilon$-accurate approximation of $\exp(tA)u_0$ expressed as a finite linear combination of tensor products of Hermite functions.
\caption{Algorithm for computing PDO semigroups on $L^2(\mathbb{R}^d)$ with error control. The details of how to do each step are provided in the proof of \cref{PDO_thm1}. The choice of Hermite functions is simply for convenience - the algorithm for other choices of bases and for different domains is analogous.}\label{alg:PDO_semigroup}
\end{algorithm}

\begin{proof}[Proof of \cref{PDO_thm1}] The proof strategy is to reduce the computational problem to the case covered by \cref{disc_thm1}. We consider the Hermite functions
$$\setlength\abovedisplayskip{4pt}\setlength\belowdisplayskip{4pt}
\psi_m(x)=(2^{m}m!\sqrt{\pi})^{-1/2}e^{-x^2/2}H_{m}(x),\quad H_m(x)=(-1)^me^{x^2}\frac{d^m}{dx^m}e^{-x^2},\quad m\in\mathbb{Z}_{\geq 0}.
$$
As an orthonormal basis of $L^2(\mathbb{R}^d)$ we consider tensor products $
\psi_{m}=\psi_{m_1}\otimes ...\otimes \psi_{m_d},$ for $m=(m_1,...,m_d)\in\mathbb{Z}_{\geq0}^d.$ Let $(A,u_0,t)\in\Omega_{\mathrm{PDE}}$. The conditions on $A$ imply that the orthonormal basis of tensor products of Hermite functions forms a core of $A$ and $A^*$ \cite{colbrook2019foundations}. We can choose a suitable ordering (such as the hyperbolic cross \cite{lubich_qm_book}) $\pmb{m}:\mathbb{N}\rightarrow \mathbb{Z}_{\geq0}^d$ of these basis functions to identify $L^2(\mathbb{R}^d)\cong l^2(\mathbb{N})$. The point of this is that any $(A,u_0,t)$ can be represented by $(\hat A, \hat u_0,t)\in\Omega_{C_0}$ through the inner products
\begin{align}
&\langle \hat Ae_k,\hat Ae_j \rangle = \int_{\mathbb{R}^d} (A\psi_{\pmb{m}(k)})\overline{(A\psi_{\pmb{m}(j)})}\,dx,\quad \langle \hat Ae_k,e_j \rangle = \int_{\mathbb{R}^d} (A\psi_{\pmb{m}(k)})\psi_{\pmb{m}(j)}\,dx \label{PDENEED1}\\
&\langle \hat u_0,e_j \rangle = \int_{\mathbb{R}^d} u_0\psi_{\pmb{m}(j)}\,dx. \label{PDENEED3}
\end{align}
\cref{PDO_thm1} follows from \cref{disc_thm1} if we can show that the evaluation functions in $\Lambda_{C_0}$ can be computed via an algorithm from the evaluation functions in $\Lambda_{\mathrm{PDE}}$.

In \cite{colbrook2019foundations} (which considered computation of spectra), quasi-Monte Carlo numerical integration was used to show that the inner products in \eqref{PDENEED1} can be computed with error control using $\Lambda_{\mathrm{PDE}}$. To finish the proof, we need only consider the evaluation functions in \eqref{needinlemma}. Let $\delta>0$ and compute $n\in\mathbb{N}$ such that $d_n\leq \delta$. It follows that
$$\setlength\abovedisplayskip{6pt}\setlength\belowdisplayskip{6pt}
\left| \langle u_0,u_0\rangle-\langle u_0|_{[-n,n]^d},u_0|_{[-n,n]^d}\rangle\right|\leq \delta^2, \quad \left|\langle u_0,e_j\rangle-\langle u_0|_{[-n,n]^d},e_j\rangle\right|\leq \delta.
$$
Since $\delta>0$ was arbitrary, it suffices to show that for any $n\in\mathbb{N}$, suitable $f_{0,m}$ and $f_{j,m}$ in \eqref{needinlemma} can be computed for $u_0|_{[-n,n]^d}$ instead of $u_0$. But this follows from the results of \cite{colbrook2019foundations} that were used to compute the integrals in \eqref{PDENEED1} (the bound in assumption (2) of the definition of $\Omega_{\mathrm{PDE}}$ is used to truncate the relevant integrals - it is not needed for $u_0$ because of the compact support of $u_0|_{[-n,n]^d}$).
\end{proof}

The choice of Hermite functions in the proof is convenient for many practical scenarios (e.g., it generates sparse matrices for polynomial coefficients) and is chosen since it allows a very large class of coefficients to be treated, namely those satisfying assumptions (2) and (3). There are, of course, many different basis choices for $L^2(\mathbb{R}^d)$ for which similar results can be proven with correspondingly different classes of coefficients for $A$. The key point is the ability to reduce the problem to $\{\Xi_{C_0},\Omega_{C_0},l^2(\mathbb{N}),\Lambda_{C_0}\}$ through computation of the relevant integrals and apply \cref{disc_thm1}. The choice of basis can also impact the computational efficiency, either in the matrix representation of the operator itself or the solution's representation. Similarly, Hilbert spaces different to $L^2(\mathbb{R}^d)$ and computational domains different to $\mathbb{R}^d$ can be treated: examples are provided in \cref{MT_example1} and \cref{num_exams_sec}.

\section{Stable and rapidly convergent quadrature for analytic semigroups}\label{analytic_semigp_sec}

\begin{figure}
\small
  \centering
	\begin{minipage}[b]{0.48\textwidth}
    \begin{overpic}[width=\textwidth,clip]{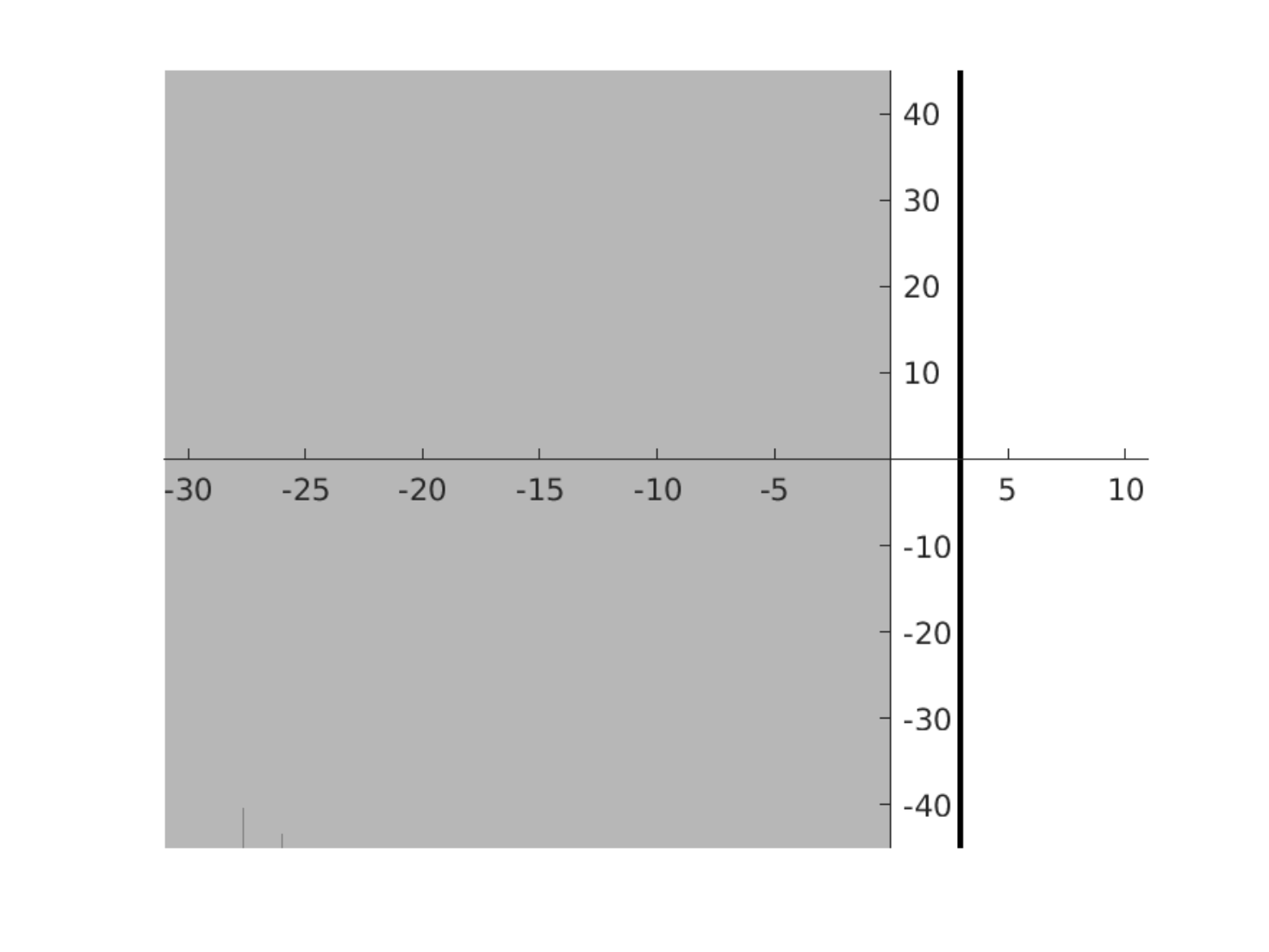}
		\put (78,59) {$\displaystyle \gamma$}
		\put (53,12) {$\displaystyle \mathrm{Im}(z)$}
		\put (82,43) {$\displaystyle \mathrm{Re}(z)$}
     \end{overpic}
  \end{minipage}
	\hfill
  \begin{minipage}[b]{0.48\textwidth}
    \begin{overpic}[width=\textwidth,trim={32mm 92mm 32mm 92mm},clip]{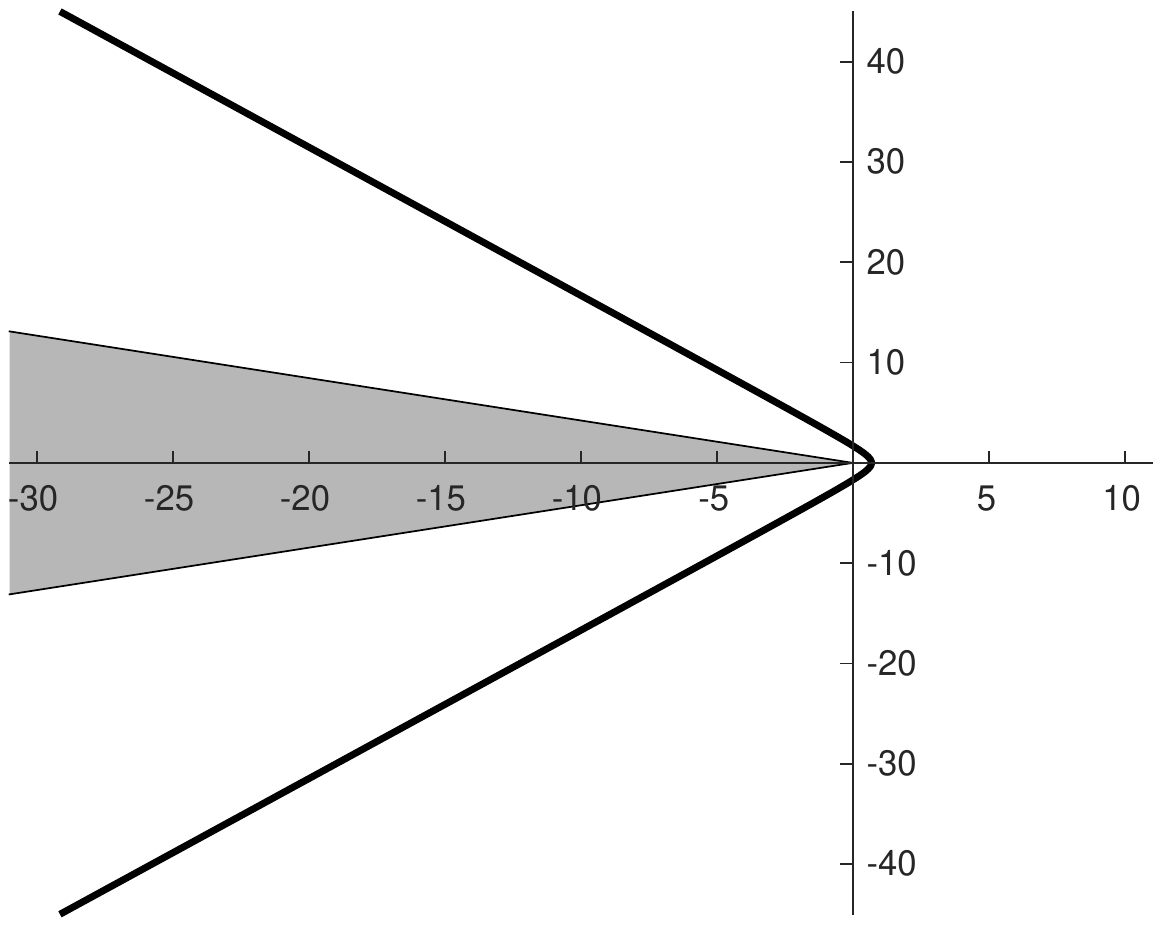}
		\put (47,59) {$\displaystyle \gamma$}
		\put (18,43) {$\displaystyle S_{\delta}^{c}$}
		\put (53,12) {$\displaystyle \mathrm{Im}(z)$}
		\put (82,43) {$\displaystyle \mathrm{Re}(z)$}
     \end{overpic}
  \end{minipage}
\caption{Left: The spectral picture of a general semigroup with $\omega=0$. The shaded region encloses the spectrum of the generator $A$ and the contour $\gamma$ is parallel to the imaginary axis. Right: Example sector containing the spectrum of a generator $A$ of an analytic semigroup (complement of $S_{\delta}$ for $\delta=0.4$) and the corresponding deformed hyperbolic contour $\gamma$.}
\label{fig:cont_examples}
\end{figure}

In this section, we assume that $A$ generates an \textit{analytic semigroup} and that the sector $S_{\delta}=\{z\in\mathbb{C}:\mathrm{arg}(z)<\pi-\delta\}$ is contained in $\rho(A)$ for some $\delta\in[0,\pi/2)$. This is shown in \cref{fig:cont_examples} (right), where we remind the reader that $\rho(A)=\mathbb{C}\backslash \mathrm{Sp}(A)$. We take the shift $\omega$ in the definition of sectorial operators to be zero without loss of generality since it can be factored back in via multiplication by a suitable exponential.

As discussed in \cref{sec_intro_sec}, there exists a large literature on quadrature of the integral appearing in \eqref{bromwich_int}. The typical approach deforms the contour of integration into a contour $\gamma$ (parametrized by $x\in\mathbb{R}$) that begins and ends in the left half-plane, such that $\mathrm{Re}(\gamma(x))\rightarrow-\infty$ as $|x|\rightarrow\infty$. This idea can be traced back to the 1950s and Talbot's doctoral student Green \cite{green1955calculation}, as well as Butcher \cite{butcher1957numerical}. Later, Talbot published a landmark paper \cite{talbot1979accurate}, where he generalized and improved the earlier work of Green. Popular contour choices include variations of Talbot's contour \cite{dingfelder2015improved}, parabolic contours \cite{gavrilyuk2001exponentially,weideman2019gauss} and hyperbolic contours such as \cref{fig:cont_examples} (right) \cite{lopez2006spectral,sheen2003parallel} (see \cite{trefethen2006talbot} for interpretations in terms of rational approximations). After such a deformation, the trapezoidal rule provides a simple and very effective quadrature rule
\begin{equation}
\label{bromwich_quad_example}
\exp(tA)u_0\approx \frac{-h}{2\pi i}\sum_{j=-N}^Ne^{z_jt}R(z_j,A)\gamma'(jh),\quad z_j=\gamma(jh).
\end{equation}
For example, \cite[Table 1]{weideman2019gauss} provides a comparison of the order of exponential convergence in the number of quadrature points for various methods in the literature.

In our setting, however, there are two important factors that must be considered. The first is the \textit{numerical stability} of the sum in \eqref{bromwich_quad_example}. Even in the unrealistic situation of computing each $R(z_j,A)$ with zero error, if $t\max(\mathrm{Re}(z_j))$ is unbounded as $N\rightarrow\infty$, then the exponential terms in \eqref{bromwich_quad_example} increase and render the sum unstable. This instability is demonstrated in the beautiful error plots of \cite{weideman2007parabolic}, which considers optimal choices of parameters for parabolic and hyperbolic contours (see also \cref{fig:new_quad_stable}).\footnote{A mechanism for providing stability for Talbot contours and for operators whose spectrum lies on the negative real axis ($\delta=0$, $0\notin\mathrm{Sp}(A)$) is given in \cite{dingfelder2015improved}. See also \cite{weideman2010improved}. Regarding hyperbolic contours, \cite{lopez2004numerical} shows weak instability for $h=\log(N)/N$ and contour parameters $\mu,\alpha$ (see \eqref{contour_choice}) independent of $N$. This leads to $\mathcal{O}(\eta\log(\log(N))+\exp({-cN/\log(N)}))$ convergence, where $\eta$ is the error in computing the resolvent. As noted \cite{lopez2006spectral}, for $N$-dependent $\mu$ and $\alpha$, stability plays a key role. A choice of $(\eta,N)$-dependent parameters is proposed in \cite{lopez2006spectral} that leads to $\mathcal{O}(\eta+\exp(-cN/\log(N)))$ convergence.} The second factor is the \textit{numerical cost of computing $R(z,A)$}. This point is largely neglected in the literature, which focuses on finite-dimensional systems (such as truncations of differential operators). Here we are referring to the cost of approximating the full infinite-dimensional $R(z,A)$, as opposed to the cost of computing $(\tilde A-zI)^{-1}$ for a truncation/discretization $\tilde A$ of $A$. For infinite-dimensional systems, the cost typically increases as $z$ approaches the spectrum of $A$ (some reasons for this are given in \cite{colbrook2020computing} in the setting of computing spectral measures). For example, in \cite{dingfelder2015improved}, the contour passes through a point on the positive real axis whose real part decreases like $\mathcal{O}(N^{-2}t^{-1})$, which leads to an increased truncation size needed to accurately compute the resolvent if $0\in\mathrm{Sp}(A)$. We therefore seek a quadrature rule that simultaneously
\begin{itemize}
	\item[(a)] avoids the growth of $t\max(\mathrm{Re}(z_j))$ as $N\rightarrow\infty$, and
	\item[(b)] whose distance to the spectrum of $A$ does not shrink as $N\rightarrow\infty$ for a fixed $t$.
\end{itemize}
\noindent{}Both points are demonstrated via numerical examples below.

We consider a hyperbolic\footnote{We do not consider parabolic contours for two reasons. First, we wish to include the case that $\delta>0$, which is impossible for parabolic contours. Second, if $\max(t\mathrm{Re}(z_j))$ is bounded above as $N\rightarrow\infty$, one can show that the optimal quadrature error when using parabolic contours is $\mathcal{O}(\exp(-c N^{2/3}))$, which is worse than the bounds we derive below.} contour parametrized as in \cite{weideman2007parabolic}:
\begin{equation}\setlength\abovedisplayskip{4pt}\setlength\belowdisplayskip{4pt}
\label{contour_choice}
\gamma(x)=\mu(1+\sin(ix-\alpha)),\quad \mu>0,\quad 0<\alpha<\frac{\pi}{2}-\delta.
\end{equation}
Since it is beneficial to reuse the computed resolvents at different times in \eqref{bromwich_quad_example}, we consider computing $\exp(tA)$ for $t\in[t_0,t_1]$ where $0<t_0\leq t_1$. Using the arguments in \cite{weideman2007parabolic}, there are three error terms associated with the choice \eqref{contour_choice}:
$$\setlength\abovedisplayskip{6pt}\setlength\belowdisplayskip{6pt}
E_1=\mathcal{O}\left(e^{-2\pi\left(\frac{\pi}{2}-\alpha-\delta\right)/h}\right),\quad
E_2=\mathcal{O}\left(e^{\mu t_1-2\pi\frac{\alpha}{h}}\right),\quad
E_3=\mathcal{O}\left(e^{\mu t_0\left(1-\sin(\alpha)\cosh(hN)\right)}\right).
$$
The first two terms correspond to the discretization error of the integral, whereas the third corresponds to the truncation error when using a finite value of $N$. We seek to \textit{asymptotically} optimize the parameters $h,\alpha$ and $\mu$, under the assumption that $\gamma(0)t_1=\mu t_1(1-\sin(\alpha))\leq\beta$ for some $\beta>0$, as $N\rightarrow\infty$. The extra free parameter $\beta$ controls the maximum size of the exponential terms in the sum \eqref{bromwich_quad_example} and is introduced to ensure stability (point (a) above).

Equating the exponential factors in $E_1$ and $E_2$ yields the equation
\begin{equation}
\label{alpha_choice}
\alpha=({h\mu t_1+\pi^2-2\pi\delta})/({4\pi}).
\end{equation}
Since $h\rightarrow 0$ as $N\rightarrow\infty$, $\alpha\rightarrow(\pi-2\delta)/4$. We therefore choose
\begin{equation}
\label{mu_choice}
\mu=(1-\sin(({\pi-2\delta})/{4}))^{-1}\beta/t_1.
\end{equation}
We then equate the exponentials in $E_2$ and $E_3$ and set $\Lambda_t=t_1/t_0$, yielding
$$\setlength\abovedisplayskip{5pt}\setlength\belowdisplayskip{5pt}
\mu(1-\Lambda_t) h+{2\pi\alpha}/{t_0}-\mu h \sin(\alpha)\cosh(hN)=0.
$$
Considering the limit $h\rightarrow 0$, we see that $hN\rightarrow \infty$ as $N\rightarrow\infty$. With the above choices, it follows that the optimal $h$ satisfies
$$\setlength\abovedisplayskip{5pt}\setlength\belowdisplayskip{5pt}
he^{hN}=\sin((\pi-2\delta)/{4})^{-1}{\pi\left({\pi-2\delta}\right)}/(t_0\mu)+\mathcal{O}(h).
$$
Using $W$ to denote the principal branch of the Lambert $W$ function, we therefore set
\begin{equation}
\label{h_choice}
h=\frac{1}{N}W\!\left(N\frac{\pi(\pi-2\delta)}{t_0\mu \sin\left(\frac{\pi-2\delta}{4}\right)}\right)\!=\frac{1}{N}W\!\left(\!\Lambda_t N\frac{\pi(\pi-2\delta)}{\beta \sin\left(\frac{\pi-2\delta}{4}\right)}\!\left(\!1-\sin\left(\frac{\pi-2\delta}{4}\right)\!\right)\!\right).
\end{equation}

\cref{alg:spec_meas} summarizes this procedure and the following theorem gives an error bound for the computed exponential. The two terms in \eqref{analytic_bound} correspond to the error in computing resolvents and the quadrature error, respectively. The error bound $\eta$ in the resolvent can be realized using the adaptive methods in \cref{disc_thm_sec} and \cref{PDE_section}, and bounds on the resolvent in a suitable sector $S_{\delta-\nu}$ \cite[Ch. 2.4]{engel1999one}. Bounds can often be found in practice, for example, by bounding the numerical range as is done via \eqref{Nrange_resbound} in \cref{num_exams_sec}. The quadrature error, which contains the constant $C$, can be made explicit by chasing the constants in \cite{weideman2007parabolic} (that appear in the $E_1,E_2$ and $E_3$ terms) and by controlling the analytic properties of the resolvent. For brevity, we have not given the explicit result. Finally, we remind the reader that \cref{an_quad_theorem} does not apply to general $C_0$-semigroups, for which one cannot deform the contour as in \cref{fig:cont_examples} and must instead apply the regularization trick of \cref{disc_thm1}.

\begin{algorithm}[t]
\textbf{Input:} $A$ (a generator of an analytic semigroup with $S_{\delta}\subset\rho(A)$ for $\delta\in[0,\pi/2)$), $u_0\in\mathcal{H}$, $0<t_0\leq t_1<\infty$, $\beta>0$, $N\in\mathbb{N}$ and $\eta>0$. \\
\vspace{-4mm}
\begin{algorithmic}[1]
\STATE Let $\gamma$ be defined as in \eqref{contour_choice} with $\alpha, \mu$ and $h$ given by \eqref{alpha_choice}, \eqref{mu_choice} and \eqref{h_choice} respectively, where $\Lambda_t=t_1/t_0$.
\STATE Set $z_j=\gamma(jh)$ and $w_j=\frac{h}{2\pi i}\gamma'(jh)$.
\STATE Solve $(A-z_jI)R_j=-u_0$ for $-N\leq j\leq N$ to an accuracy $\eta$.
\end{algorithmic} \textbf{Output:} $\tilde{u}_N(t)=\sum_{j=-N}^Ne^{z_jt}w_jR_j$ for $t\in[t_0,t_1]$.
\caption{Stable and rapidly convergent algorithm for analytic semigroups.}\label{alg:spec_meas}
\end{algorithm}

\begin{theorem}\label{an_quad_theorem}
Suppose that $A$ generates an analytic semigroup and $S_{\delta}\subset \rho(A)$ for $\delta\in[0,\pi/2)$. Let $u_0\in\mathcal{H}$, $0<t_0\leq t_1<\infty$, $\beta>0$ and $\eta>0$. Let $\tilde{u}_N(t)$ denote the output of \cref{alg:spec_meas} for $N\in\mathbb{N}$ and $a=\sin(\pi/4-\delta/2)^{-1}-1$. Then there exists a constant $C$ such that the following bound holds for any $t_0\leq t\leq t_1$ as $N\rightarrow\infty$,
\begin{equation}
\label{analytic_bound}
\begin{split}
\left\| \exp(tA)u_0-\tilde{u}_N(t)\right\|\!&\leq \! \underbrace{\left({2\mu e^{\frac{\beta}{1-\sin(\alpha)}}}{\pi^{-1}}   \!\!\int_0^\infty \!\!\!\! e^{x-\mu t\sin(\alpha)\cosh(x)}dx\right) \eta}_{\text{numerical error due to inexact resolvent}} \\
&\quad\quad\quad\quad+\underbrace{C e^{\frac{\beta}{1-\sin(\alpha)}}\cdot\exp\left(\!-\frac{N\pi(\pi-2\delta)/2}{\log(\Lambda_t  \frac{a}{\beta}N\pi(\pi-2\delta))}\right)}_{\text{quadrature error}}.
\end{split}
\end{equation}
\end{theorem}

\begin{proof}
Let $t_0\leq t\leq t_1$. The analysis preceding the theorem shows that
$$
\max\{E_1,E_2\}=\mathcal{O}(\exp((1-\sin(\pi/4-\delta/2))^{-1}\beta-2\pi{\alpha}/{h}).
$$
Using $\sin(\alpha)>\sin(\pi/4-\delta/2)$ for small $h$ and the definition of $a$, we see that
\begin{equation}
\label{E1E2bound}
\max\{E_1,E_2\}=\mathcal{O}\left(\exp\! \left(\! \frac{\beta}{1-\sin(\alpha)}\! -\! \frac{N\pi(\pi-2\delta)/2}{W\!(\Lambda_t N \frac{a}{\beta}\pi(\pi-2\delta))}\! \right)\right).
\end{equation}
Using $\cosh(x)\geq e^x/2$, and $W(z)e^{W(z)}=z$, we have that
$$
\mu t_0\sin(\alpha)\cosh(hN)\geq \frac{N\pi(\pi-2\delta)/2}{W\!(\Lambda_t N \frac{a}{\beta}\pi(\pi-2\delta))}\frac{\sin(\alpha)}{\sin(\pi/4-\delta/2)}\geq \frac{N\pi(\pi-2\delta)/2}{W\!(\Lambda_t N \frac{a}{\beta}\pi(\pi-2\delta))}.
$$
Since $t_0\leq t_1$, it follows that $E_3$ satisfies the same bound as $E_1$ and $E_2$ in \eqref{E1E2bound}. Using $\|\cdot\|$ to denote the operator norm, it follows that
\begin{equation*}\setlength\abovedisplayskip{5pt}\setlength\belowdisplayskip{5pt}
\label{an_bound1}
\Big\| \exp(tA) + \frac{h}{2\pi i}\sum_{j=-N}^Ne^{z_jt}\gamma'(jh)R(z_j,A)\Big\|\! \leq\!  C_1\!\exp\! \left(\!\! \frac{\beta}{1-\sin(\alpha)}\! -\! \frac{N\pi(\pi-2\delta)/2}{W\!(\Lambda_t N \frac{a}{\beta}\pi(\pi-2\delta))}\!\! \right)\! ,
\end{equation*}
for a constant $C_1$. Since $\|[R_j\!+\! R(z_j\! ,A)]u_0\|\!\leq\! \eta$, the left-hand side of \eqref{analytic_bound} is bounded by
\begin{equation}\setlength\abovedisplayskip{4pt}\setlength\belowdisplayskip{5pt}
\label{an_bound2}
\frac{\eta}{2\pi}\!\!\sum_{j=-N}^Nh|e^{z_jt}\gamma'(jh)|+C_1e^{\frac{\beta}{1-\sin(\alpha)}}\exp\!\left(\!-\frac{N\pi(\pi-2\delta)/2}{W\!(\Lambda_t N \frac{a}{\beta}\pi(\pi-2\delta))}\!\right).
\end{equation}
A simple calculation shows that $|\exp(\gamma(x)t)\gamma'(x)|\leq \mu \exp(|x|+\mu t-\mu t \sin(\alpha)\cosh(x))$. For notational convenience, set $c=\mu t\sin(\alpha)$. Using this bound in \eqref{an_bound2}, the resulting sum for non-negative $j$ is a Riemann sum for the integral in \eqref{analytic_bound} (up to constants). The region of integration can be split into two regions $\{0\leq x <\sinh^{-1}(c^{-1})\}$ and $\{x> \sinh^{-1}(c^{-1})\}$, on which the integrand has positive and negative derivatives respectively. Let $Jh\leq \sinh^{-1}(c^{-1})\leq (J+1)h$ for $J\in \mathbb{N}$. In a neighborhood of the turning point $\sinh^{-1}(c^{-1})$, the integrand is concave and hence for small $h$ we can also bound the sum from $j=J$ to $j=J+1$ by the integral over $[(J-1)h,(J+2)h]$. It follows that
$$\setlength\abovedisplayskip{5pt}\setlength\belowdisplayskip{5pt}
\sum_{j=-N}^Nh|e^{z_jt}\gamma'(jh)|\leq {4\mu e^{\frac{\beta}{1-\sin(\alpha)}}}\int_{0}^\infty e^{x-c\cosh(x)}dx,\quad \text{for large $N$.}
$$
Bounding the second term of \eqref{an_bound2} using $W\!(x)\geq \log(x)$ for large $x$ yields \eqref{analytic_bound}.
\end{proof}

Smaller $\beta$ leads to a smaller error plateau for a given $\eta$ as $N\rightarrow\infty$, but a slower rate of convergence (with the rate depending only logarithmically on $\beta$). In practice, we found that results were not strongly dependent on reasonable choices of $\beta$ and have used a default value of $\beta=3$ in what follows.

\begin{figure}\vspace{-1mm}
  \centering
  \begin{minipage}[b]{0.48\textwidth}
    \begin{overpic}[width=\textwidth]{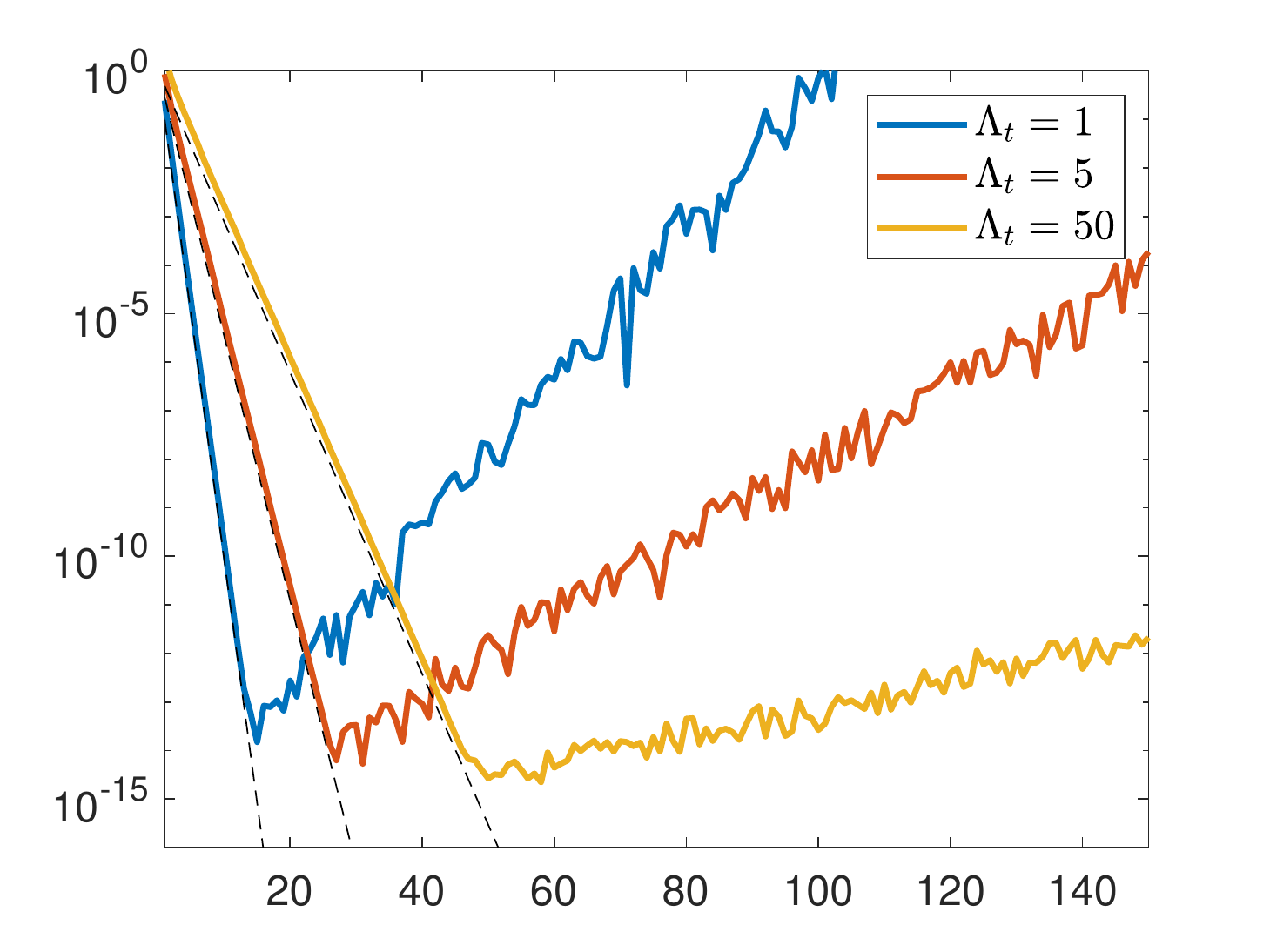}
     \put (18,73) {parameter choices from \cite{weideman2007parabolic}}
		\put(50,-2) {$N$}
		\put(2,37) {\rotatebox{90}{$M_N$}}
     \end{overpic}
  \end{minipage}
  \hfill
  \begin{minipage}[b]{0.48\textwidth}
    \begin{overpic}[width=\textwidth]{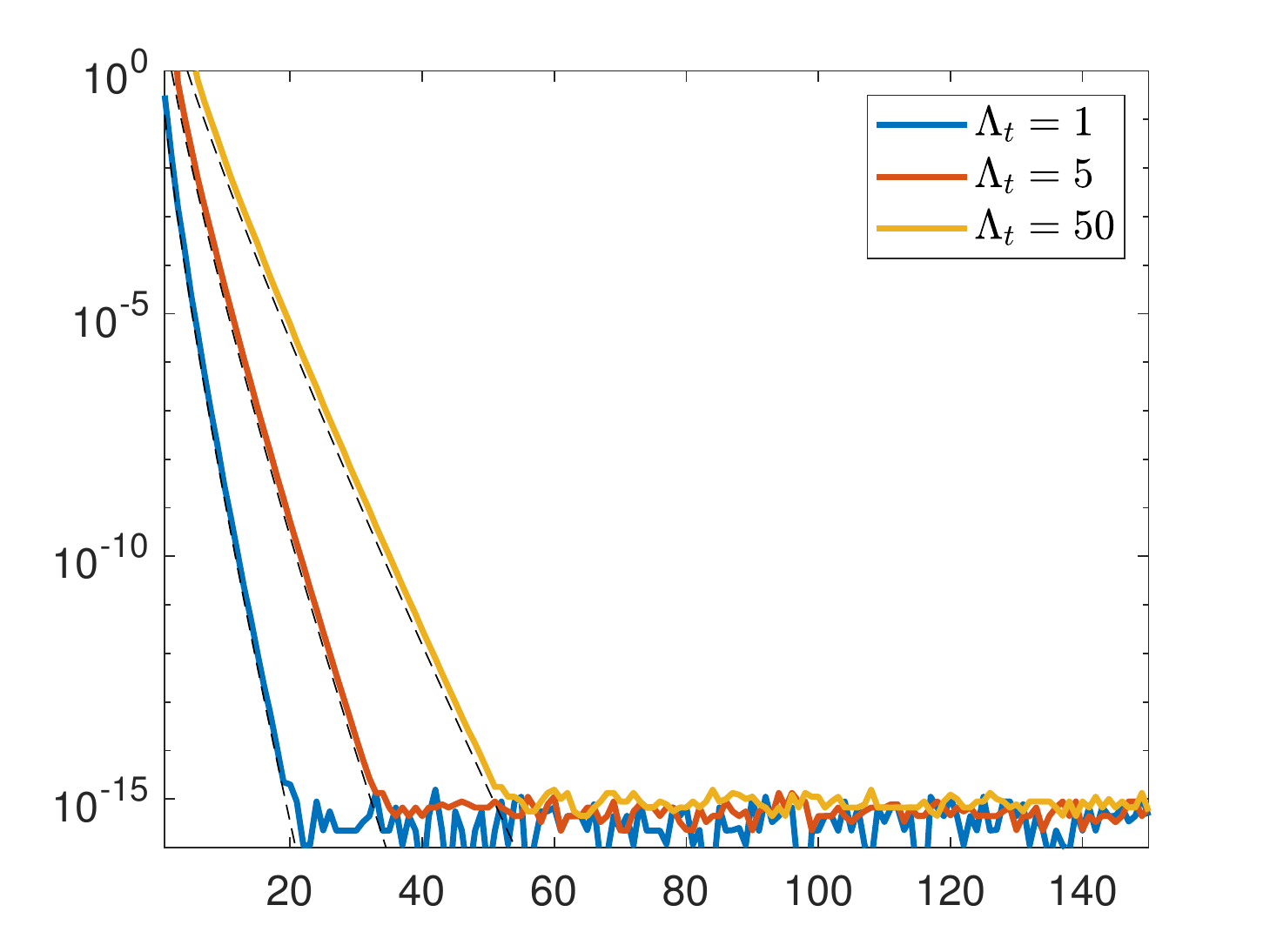}
    \put (21,73) {proposed quadrature rule}
		\put(50,-2) {$N$}
		\put(2,37) {\rotatebox{90}{$M_N$}}
     \end{overpic}
  \end{minipage}
  \vspace{-2mm}\caption{Left: Maximum error $M_N$ (defined in \eqref{M_err_def}) for the optimal hyperbolic contour of \cite{weideman2007parabolic}, showing instability for large $N$. Right: Same but for the parameter choices from \cref{an_quad_theorem} and $\beta=3$, showing stability for large $N$. In both cases, the dashed lines represent the theoretical convergence rates.}
\label{fig:new_quad_stable}
\end{figure}
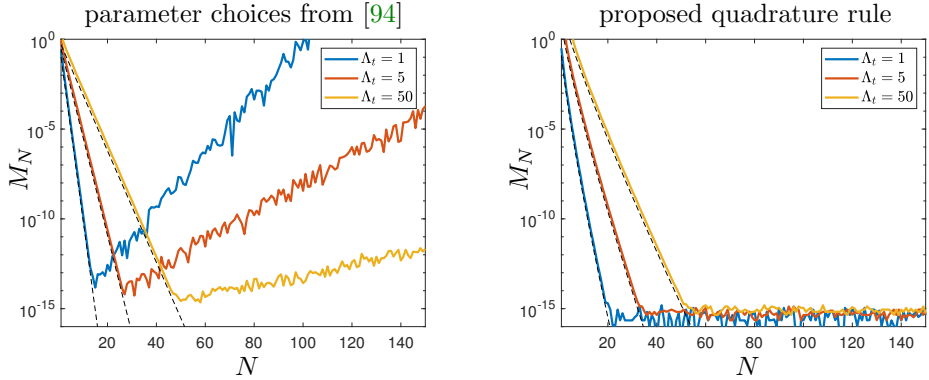

\subsection{Numerical example showing stability} As a simple example to demonstrate the numerical stability of the proposed quadrature rule, we consider
$$\setlength\abovedisplayskip{6pt}\setlength\belowdisplayskip{6pt}
e^{\lambda t}=\frac{1}{2\pi i}\int_{\gamma}\frac{e^{zt}}{z-\lambda}dz,\quad \lambda\leq  0.
$$
To model the situation of singularities at $0$ corresponding to $0\in\mathrm{Sp}(A)$, we take $\lambda =0$ (similar behavior occurs for other choices) and consider the maximum error
\begin{equation}\setlength\abovedisplayskip{6pt}\setlength\belowdisplayskip{6pt}
\label{M_err_def}
M_N=\max_{t\in[t_0,t_1]}\left|1-Q_{N}(t)\right|,
\end{equation}
where $Q_N$ denotes the output of the quadrature rule under consideration. \cref{fig:new_quad_stable} shows the error for $t_0=0.1$ and a range of $\Lambda_t$ values for the quadrature rule in \cite{weideman2007parabolic} (using the optimal values of parameters in Table 1 of \cite{weideman2007parabolic}) and the quadrature rule of \cref{an_quad_theorem}. In both cases, the dashed lines represent the theoretical convergence rate estimates and agree well with the convergence seen in practice. We see the instability of the quadrature rule in \cite{weideman2007parabolic}. Such instabilities can be a problem for \eqref{bromwich_int}, where estimating the optimal value of $N$ and optimal contour parameters can be difficult for an operator $A$ with an unknown spectrum (corresponding to unknown singularities of the resolvent). This can be difficult even in the scalar case when singularities can be analyzed (\cite[Fig. 1]{garrappa2013evaluation} shows an example of loss of accuracy), and typically requires the user to perform substantial case-specific analysis \cite{pang2016fast}. In contrast, the quadrature rule of \cref{an_quad_theorem} is completely stable. An advantage of a stable quadrature rule is that it no longer becomes essential to determine the optimal value of $N$ and optimal contour parameters. Increasing $N$ no longer incurs a stability penalty. Moreover, within the stable region for the quadrature rule in \cite{weideman2007parabolic}, there is very little difference between the convergence rates of each method.

\subsection{Numerical example demonstrating \eqref{analytic_bound}}\label{MT_example1} To demonstrate the bound \eqref{analytic_bound} and the effects of computing the resolvent near the spectrum of an infinite-dimensional operator, we consider the variable diffusion equation on $L^2(\mathbb{R})$
\begin{equation}
\label{variable_diff_example}
u_t=[(1.1-{1}/({1+x^2}))u_x]_x,\quad u_0(x)=e^{-\frac{(x-1)^2}{5}}\cos(2x)+{2}[{1+(x+1)^4}]^{-1}.
\end{equation}
To represent the operator, we use the Malmquist--Takenaka functions, defined by
$$
\phi_n(x)=\sqrt{{L}/{\pi}}{(1+iLx)^n}{(1-iLx)^{-(n+1)}},\quad n\in\mathbb{Z},
$$
where $L>0$ denotes a scaling factor (we take $L=1/5$, which has not been optimized). Following the remarks after \cref{PDO_thm1}, we can still compute the relevant evaluation functions $\Lambda$ for this basis. The orthonormal basis $\{\phi_n\}_{n\in\mathbb{Z}}$ has an interesting history \cite{iserles2020family} and approximation properties can be found in \cite{boyd1987spectral,weideman1994computation}. We rapidly compute expansions in this basis using the fast Fourier transform. The basis also simultaneously provides sparse matrices for differentiation and multiplication. Even for the most computationally expensive examples in this section (those with the largest number of quadrature points and basis functions), the computation (including forming the linear system, solving for all quadrature points etc.) took less than half a second on a modest laptop without parallelization.

\begin{figure}\vspace{-1mm}
  \centering
  \begin{minipage}[b]{0.48\textwidth}
    \begin{overpic}[width=\textwidth]{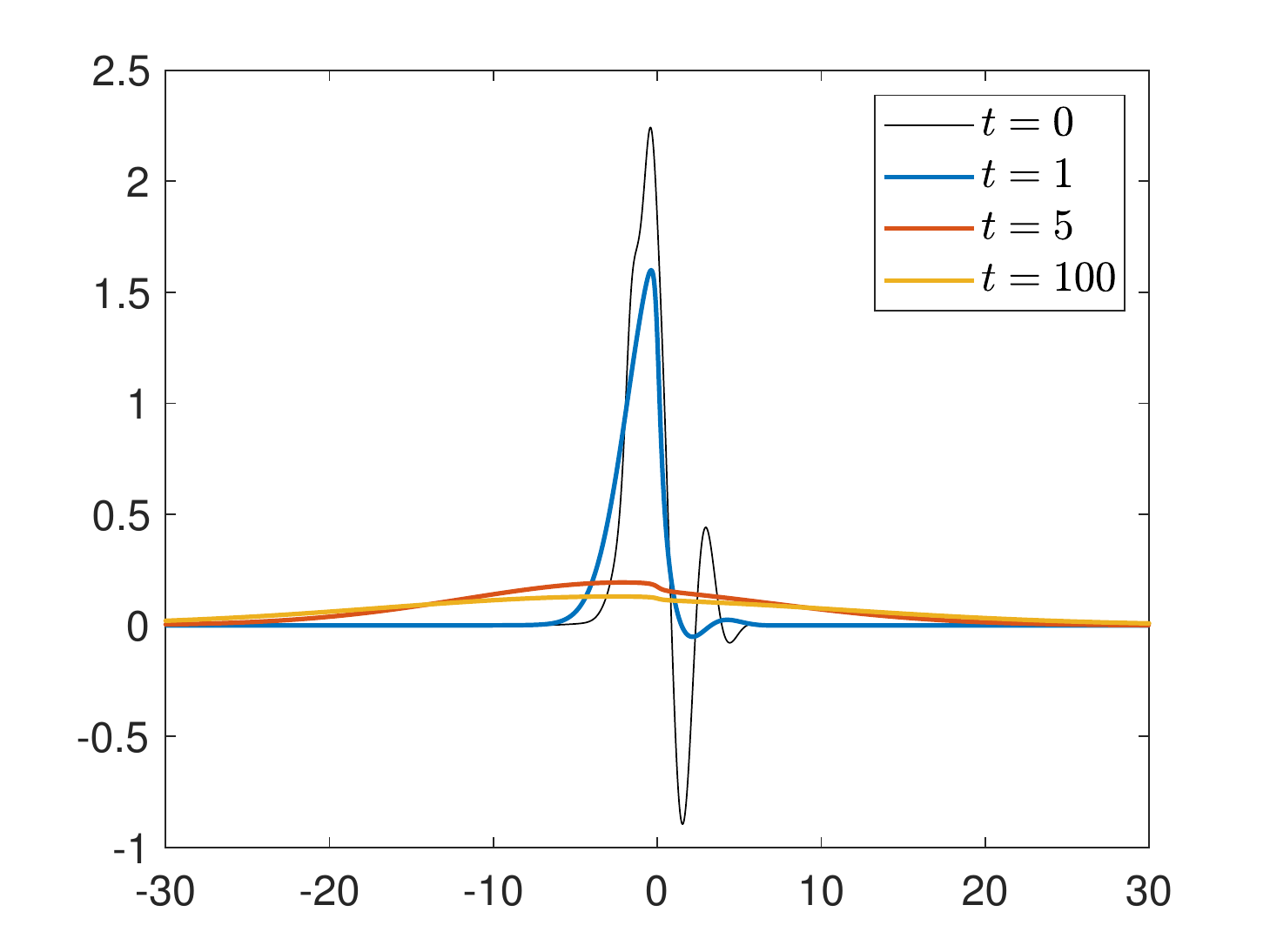}
     \put (17,73) {solutions at different time $t$}
		\put(50,-2) {$x$}
     \end{overpic}
  \end{minipage}
  \hfill
  \begin{minipage}[b]{0.48\textwidth}
    \begin{overpic}[width=\textwidth]{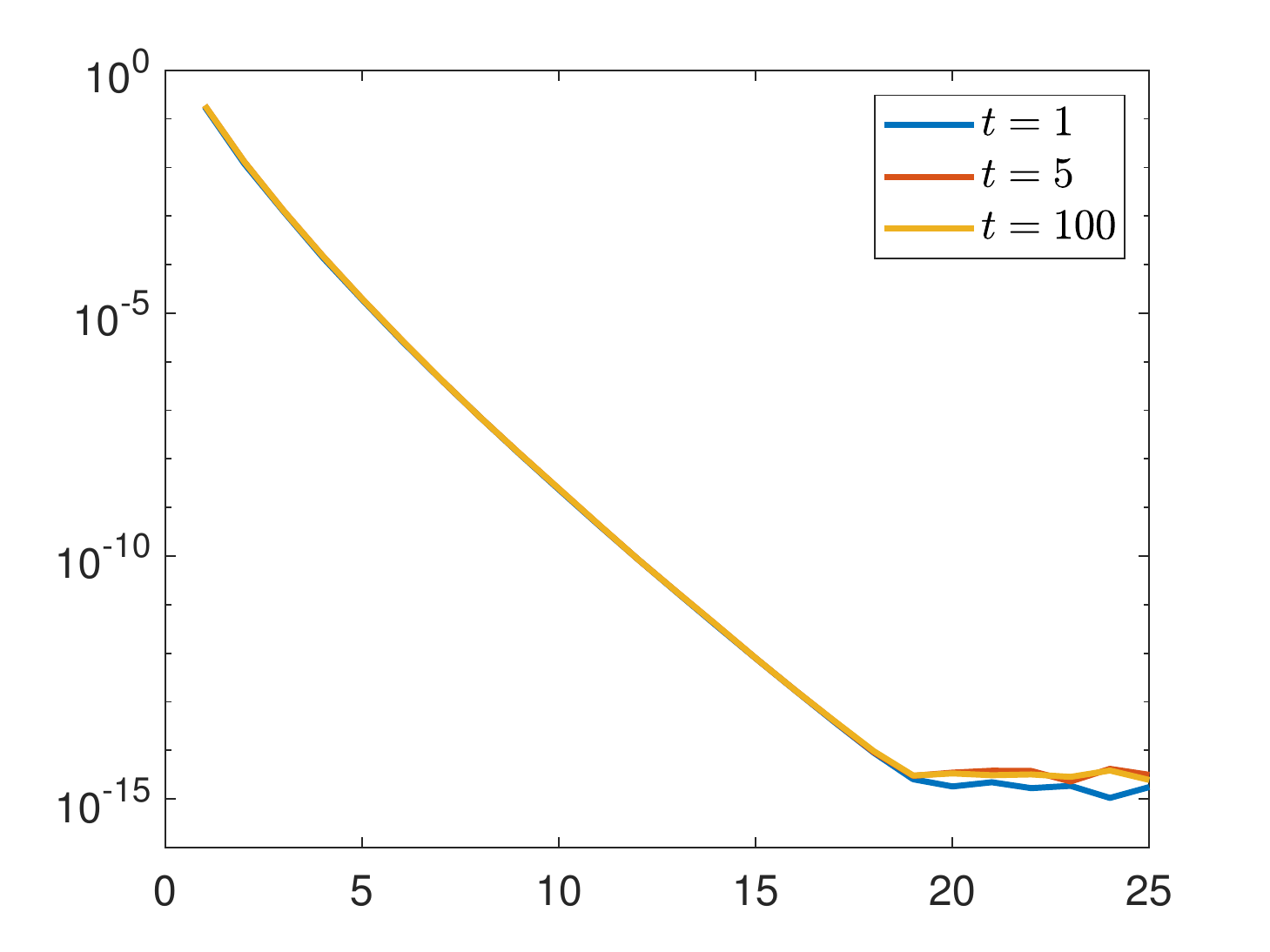}
    \put (28,73) {relative errors of $u$}
		\put(50,-2) {$N$}
     \end{overpic}
  \end{minipage}
  \vspace{-1mm}\caption{Left: Computed solutions of the variable diffusion equation \eqref{variable_diff_example}, with an error bound $\epsilon=10^{-12}$. Right: Relative errors of the computed solutions in terms of $N$, where resolvents $R(z_j,A)$ are computed with an adaptive number of basis functions so that the second term on the right-hand side of \eqref{analytic_bound} dominates the $\eta$ term. Errors are computed in the $L^2(\mathbb{R})$ norm.}\label{fig:new_quad_convergence}
\end{figure}

\cref{fig:new_quad_convergence} (left) shows computed solutions with a rigorous error bound $\epsilon=10^{-12}$. As a first demonstration of \eqref{analytic_bound}, we compute the resolvents so that the second term on the right-hand side of \eqref{analytic_bound} dominates the $\eta$ term. \cref{fig:new_quad_convergence} (right) shows the convergence (where errors are computed by comparison with larger parameter values to approximate the true error as opposed to a computed bound) in the number of quadrature points for $\Lambda_t=1$. The error is almost independent of $t$ as expected from \eqref{analytic_bound} (recall that $\mu t\leq \beta$).

The situation becomes quite different when we consider the first error term on the right-hand side of \eqref{analytic_bound}. \cref{fig:new_quad_convergence2} shows errors in terms of $m$, where we take $\{\phi_n:|n|\leq m\}$ as our basis functions and fixed $N=30$ quadrature points so that the first term on the right-hand side of \eqref{analytic_bound} dominates. The dashed lines in the left panel show relative errors in the computed resolvents $R(\gamma(0),A)$, where $A$ is the generator. As $t$ increases, more basis functions are needed to compute the resolvent to a given accuracy $\eta$ since $\gamma(0)\sim t^{-1}$ decreases to $0\in\mathrm{Sp}(A)$. However, we can still compute accurate solutions for large time using a modest number of basis functions. In this example, the cost per solve scales as $\mathcal{O}(m)$ owing to the banded matrix representation. In general, the bandwidth depends on the regularity of the coefficients of the PDO. For this example and in \cref{complex_diff_example,Elastic_example}, we do not include this when reporting the complexity in terms of $m$. The solid lines in \cref{fig:new_quad_convergence2} (right) show the errors of the full approximation, which decrease at roughly the same rate as $m\rightarrow\infty$.

\begin{figure}\vspace{-1mm}
  \centering
  \begin{minipage}[b]{0.48\textwidth}
    \begin{overpic}[width=\textwidth]{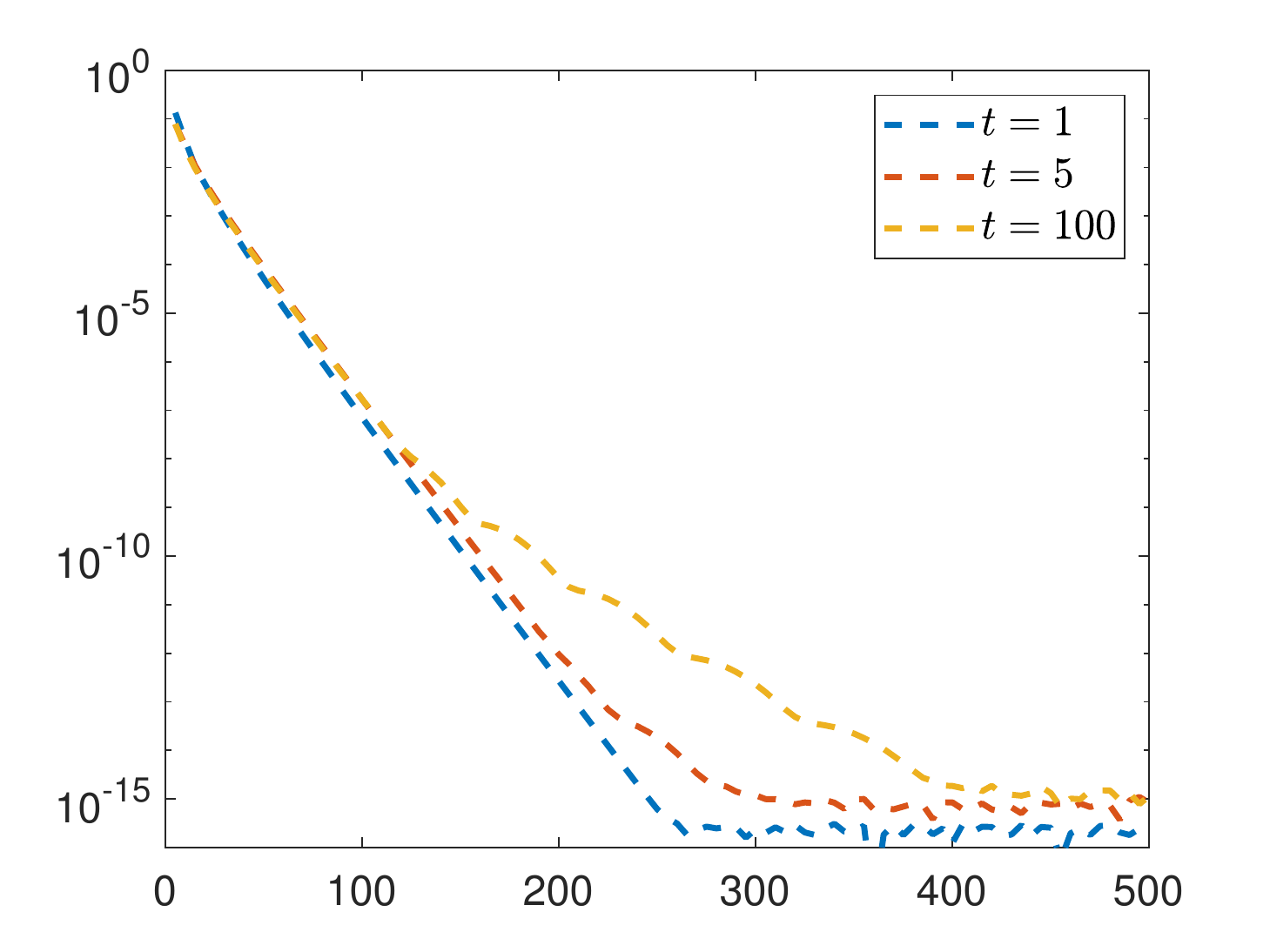}
    \put (18,73) {relative errors of $R(\gamma(0),A)$}
		\put(50,-2) {$m$}
     \end{overpic}
  \end{minipage}
  \hfill
  \begin{minipage}[b]{0.48\textwidth}
    \begin{overpic}[width=\textwidth]{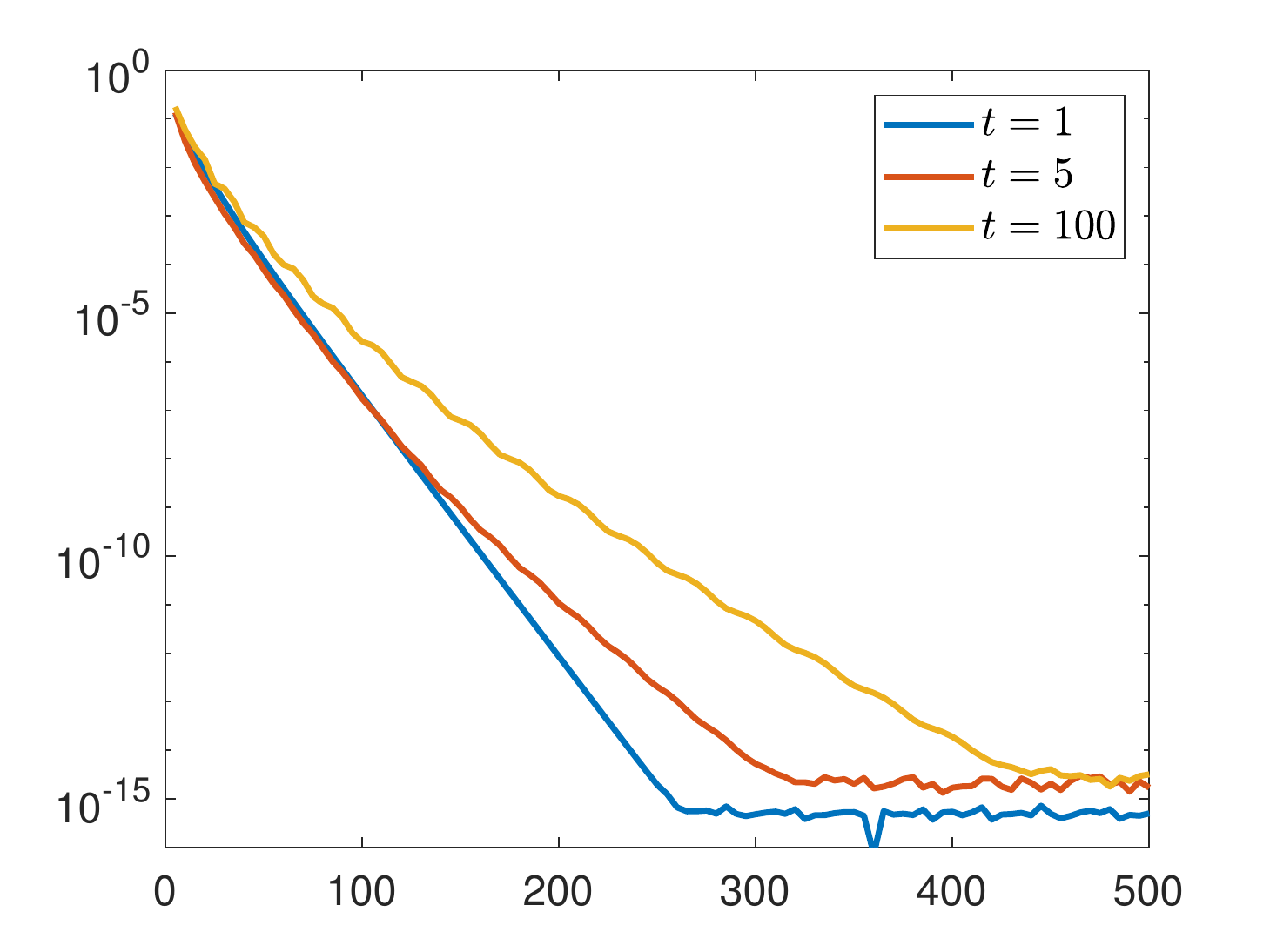}
    \put (28,73) {relative errors of $u$}
		\put(50,-2) {$m$}
     \end{overpic}
  \end{minipage}
  \vspace{-2mm}\caption{Left: Relative errors of the computed resolvents $R(\gamma(0),A)$ in terms of $m$ (the number of basis functions used is $2m+1$). Right: Relative errors of the computed solutions $u$ in terms of $m$. Errors are computed in the $L^2(\mathbb{R})$ norm. The number of quadrature points ($2N+1$ with $N=30$) is chosen such that the first term on the right-hand side of \eqref{analytic_bound} dominates.}\label{fig:new_quad_convergence2}\vspace{-4mm}
\end{figure}

\section{Numerical examples}\label{num_exams_sec}
We now provide three examples to demonstrate the versatility of our approach: an infinite discrete system, a PDE on an unbounded domain, and a second-order (in time) PDE on a bounded domain. To gain error bounds for non-normal generators $A$, we bound the resolvent\footnote{See also \cite{weideman2010improved} for pseudospectra considerations when choosing contours.} and spectrum using
\begin{equation}\setlength\abovedisplayskip{7pt}\setlength\belowdisplayskip{7pt}\label{Nrange_resbound}
\mathrm{Sp}(A)\subset \overline{\mathcal{N}(A)}\cup \overline{\mathcal{N}(A^*)},\quad \|R(z,A)\|\leq [\mathrm{dist}(z,\overline{\mathcal{N}(A)})]^{-1}\hspace{2mm}\forall z\notin \overline{\mathcal{N}(A)}\cup \overline{\mathcal{N}(A^*)},
\end{equation}
where $\mathcal{N}(A):=\{\langle Ax,x\rangle:x\in\mathcal{D}(A),\|x\|=1\}$ denotes the numerical range. Working in infinite dimensions has the advantage that it is often \textit{much easier} to obtain such bounds for $A$ than for a discretization or truncation of $A$. In what follows we use floating-point arithmetic to compute solutions and error bounds, though it is simple to adapt the algorithms using interval arithmetic \cite{tucker2011validated,rump2010verification}. Bounds only need to be verified for residuals of the solved linear systems. For example, this could be useful in the growing area of computer-assisted proofs \cite{fefferman1990,fefferman1996interval,Hales_Annals}.

\subsection{Schr\"odinger and wave equations on the Ammann--Beenker tiling}\label{AB_example}

In this example, we consider semigroups on the aperiodic Ammann--Beenker (AB) tiling, a model of a 2D quasicrystal (aperiodic crystals), shown in \cref{fig:AB_1} (left). Quasicrystals were discovered in 1982 by Shechtman \cite{PhysRevLett.53.1951} who was awarded the Nobel prize in 2011 for his discovery. They have generated considerable interest because of their exotic physical and spectral properties \cite{stadnik2012physical}. Whilst physical transport properties of quasicrystals are well-studied \cite{roche1997electronic,janssen2018aperiodic}, computing the relevant semigroups on the infinite-dimensional space is challenging because of the aperiodicity. For example, one cannot use absorbing boundary conditions. We consider the graph Laplacian
$$\setlength\abovedisplayskip{5pt}\setlength\belowdisplayskip{5pt}
[\Delta_{\mathrm{AB}}\psi]_{i} = \sum_{i \sim j} \left(\psi_j-\psi_i\right),
$$
with summation over nearest neighbor vertices. A natural ordering of the vertices leads to an operator acting on $l^2(\mathbb{N})$ whose local bandwidth grows. We consider the (free) Schr\"odinger equation and the wave equation given by
$$\setlength\abovedisplayskip{5pt}\setlength\belowdisplayskip{5pt}
iu_t=-\Delta_{\mathrm{AB}}u\quad \text{and}\quad u_{tt}=\Delta_{\mathrm{AB}}u,
$$
respectively. As our initial condition, we take $u_0$ to be $1$ on the vertices shown as red dots in \cref{fig:AB_1}, and zero otherwise (for the wave equation, we set $u'(0)=0$).

\begin{figure}\vspace{-1mm}
  \centering
  \begin{minipage}[b]{0.48\textwidth}
    \begin{overpic}[width=\textwidth,trim={25mm 90mm 25mm 90mm},clip]{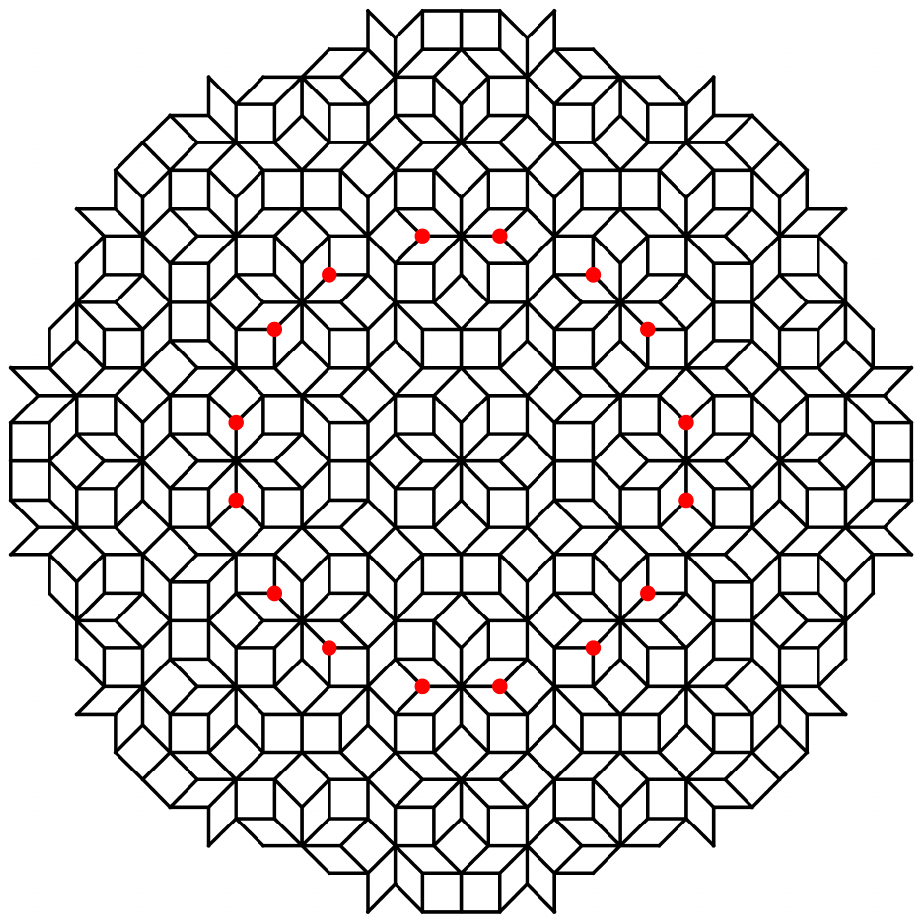}
     \end{overpic}
  \end{minipage}
	\begin{minipage}[b]{0.48\textwidth}
    \begin{overpic}[width=\textwidth]{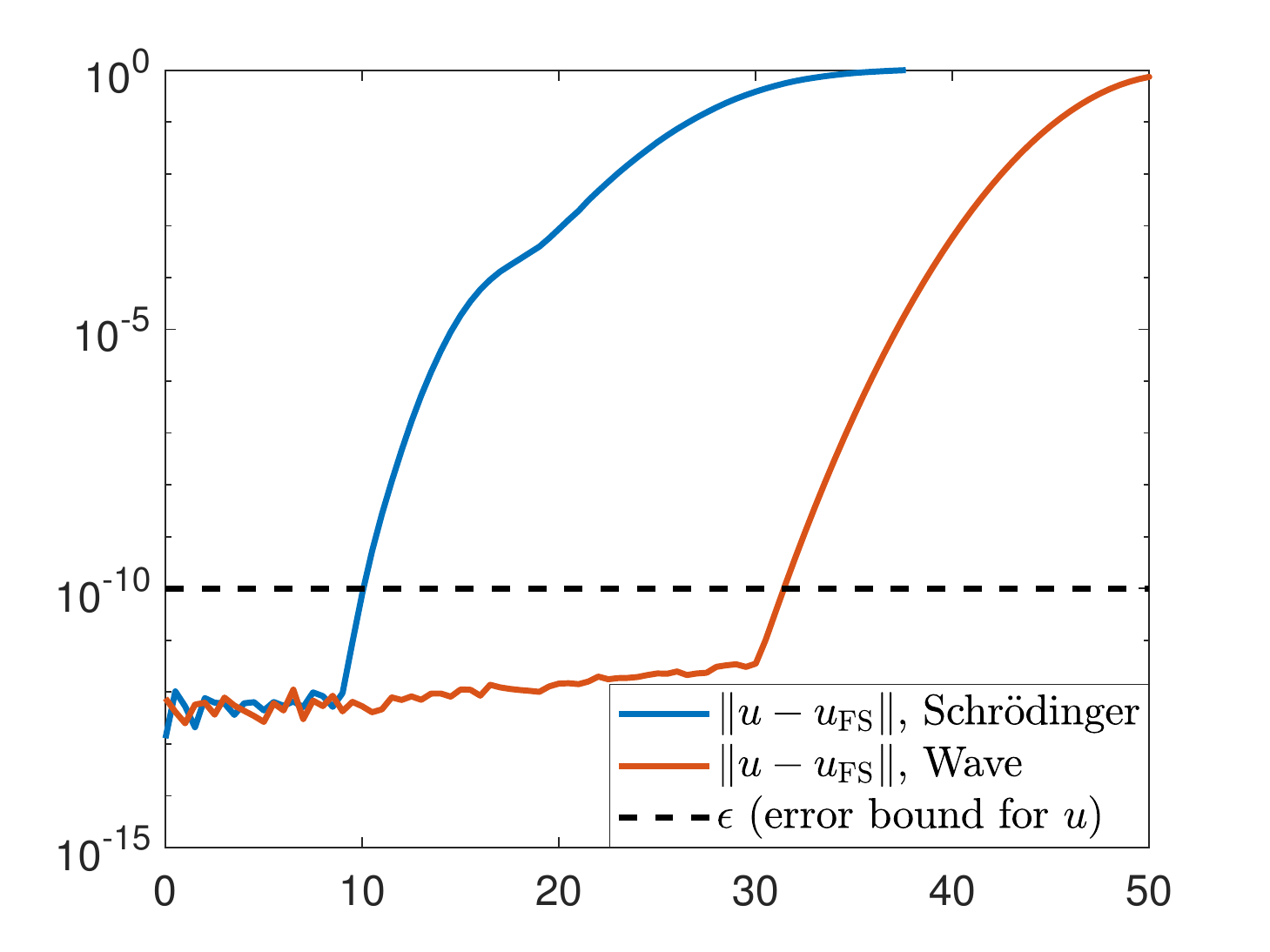}
		\put(50,-1) {$t$}
     \end{overpic}
  \end{minipage}

  \vspace{-2mm}\caption{Left: Finite portion of the infinite aperiodic Ammann--Beenker (AB) tiling, generated from an incommensurate rotation and projection of the 4D hypercubic lattice. The AB tiling is aperiodic with global 8-fold rotational symmetry around a central point, and the lattice is obtained by considering all vertices of the tiling. The red dots correspond to $u_0$ (see main text). Right: Difference in norm between the solution computed using the method of this paper ($u$, which has a guaranteed error bound of $10^{-10}$) and the solution computed using Galerkin projection ($u_{\mathrm{FS}}$). As $t$ increases, the difference begins to grow quickly as $u_{\mathrm{FS}}$ becomes inaccurate because of boundary effects.}\label{fig:AB_1}
\end{figure}

\begin{figure}
  \centering
  \begin{minipage}[b]{0.32\textwidth}
    \begin{overpic}[width=\textwidth]{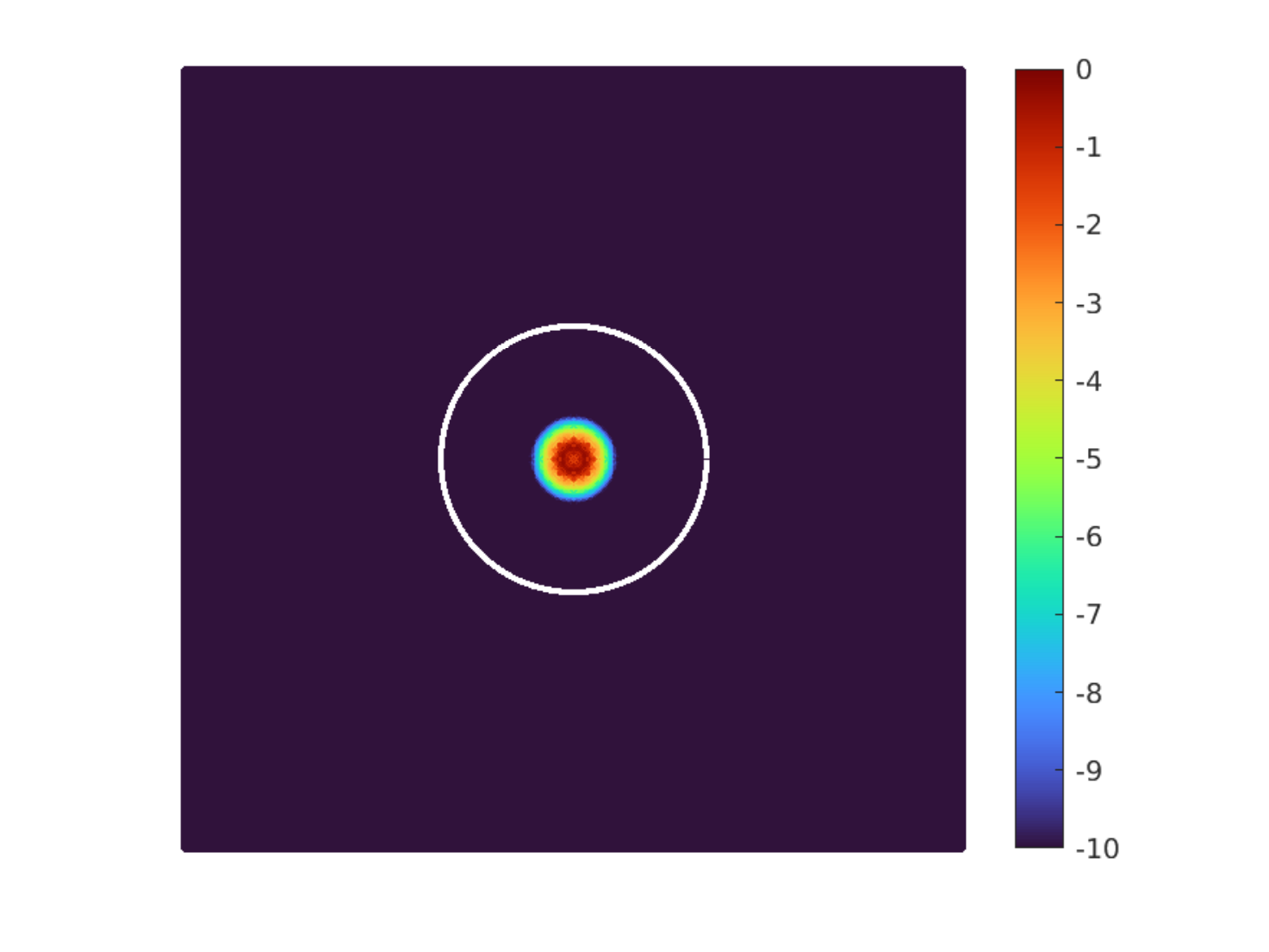}
     \end{overpic}
  \end{minipage}
	\begin{minipage}[b]{0.32\textwidth}
    \begin{overpic}[width=\textwidth]{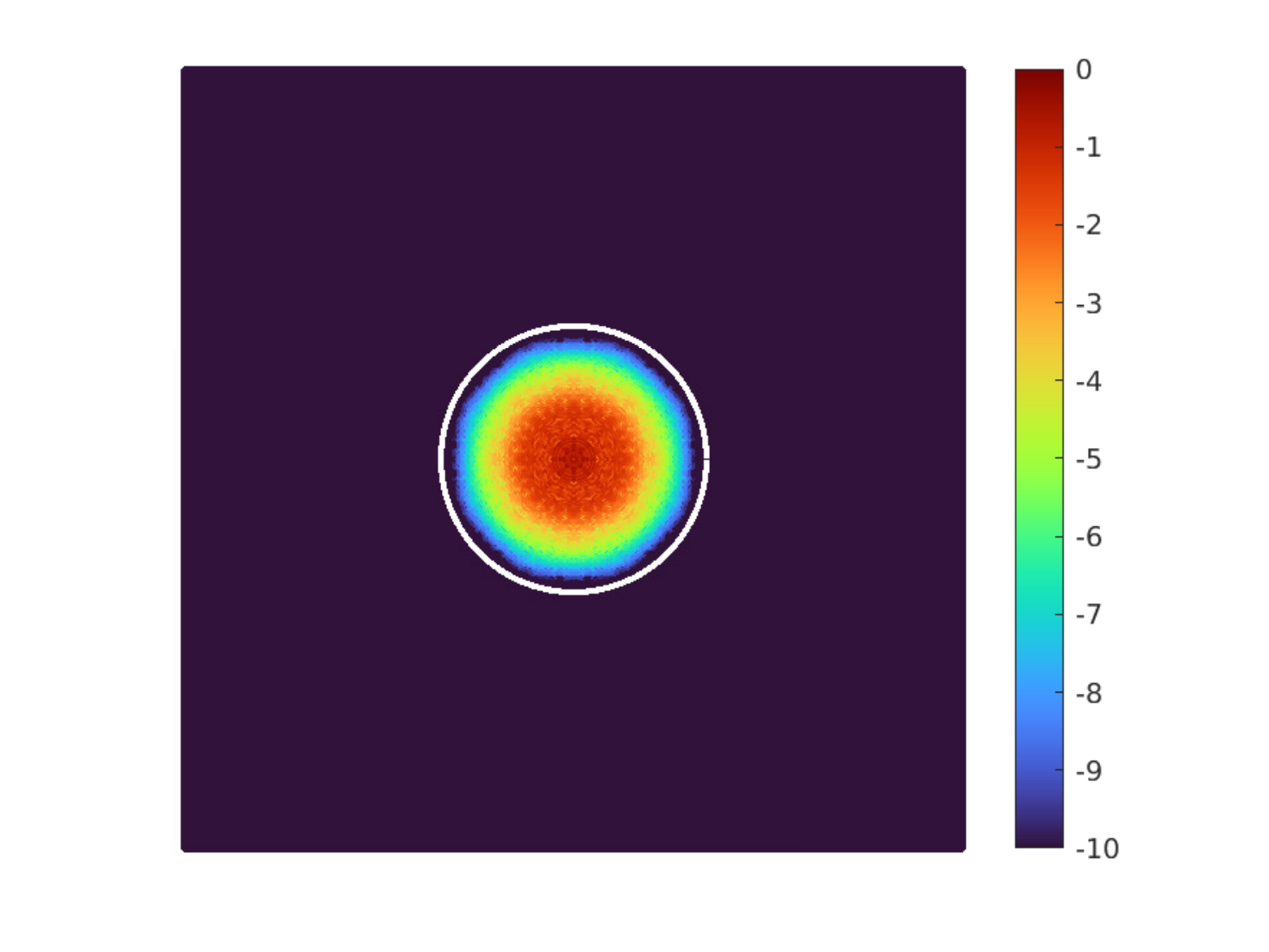}
     \end{overpic}
  \end{minipage}
	\begin{minipage}[b]{0.32\textwidth}
    \begin{overpic}[width=\textwidth]{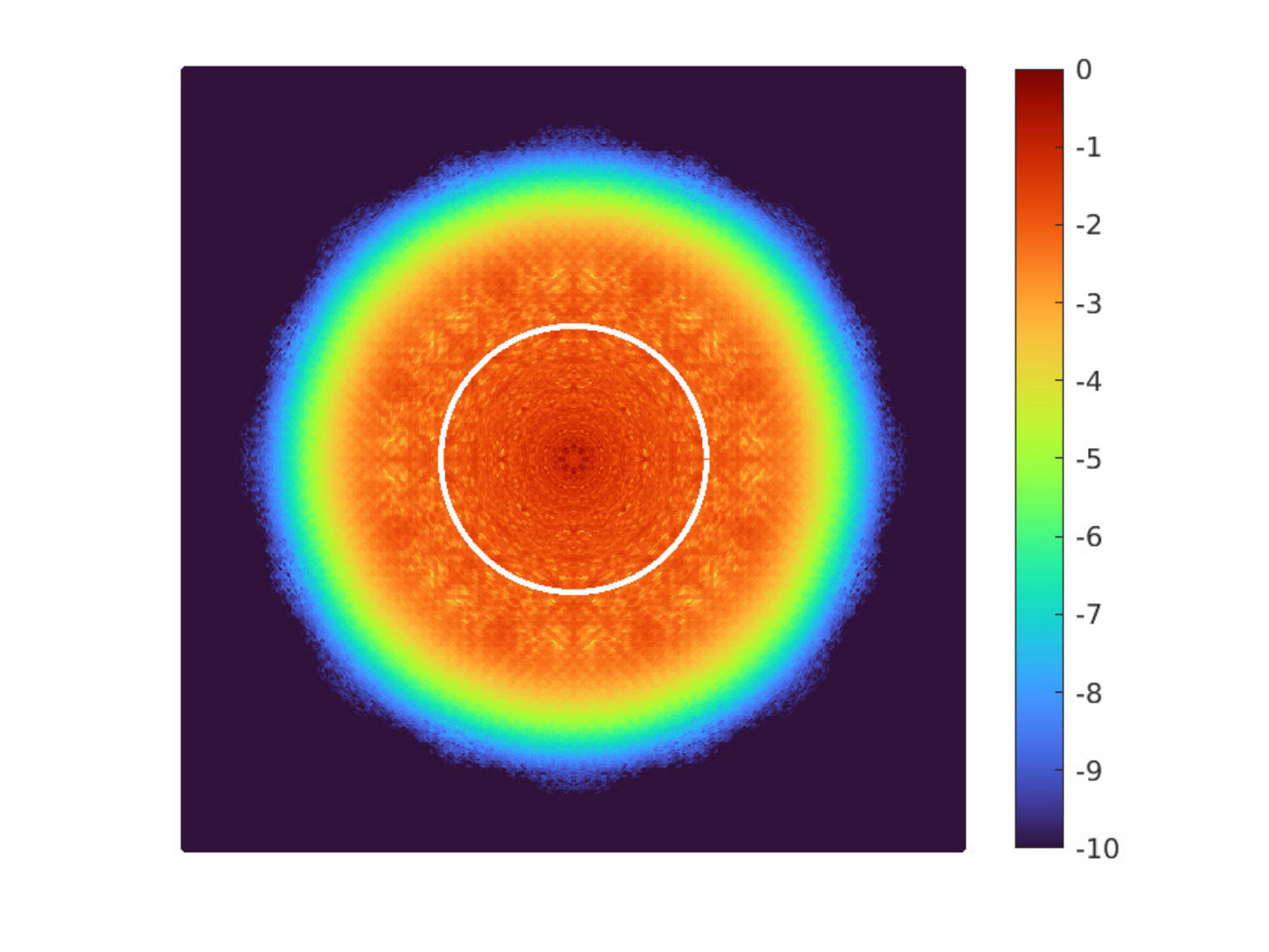}
     \end{overpic}
  \end{minipage}
	
	\begin{minipage}[b]{0.32\textwidth}
    \vspace{-4mm}\begin{overpic}[width=\textwidth]{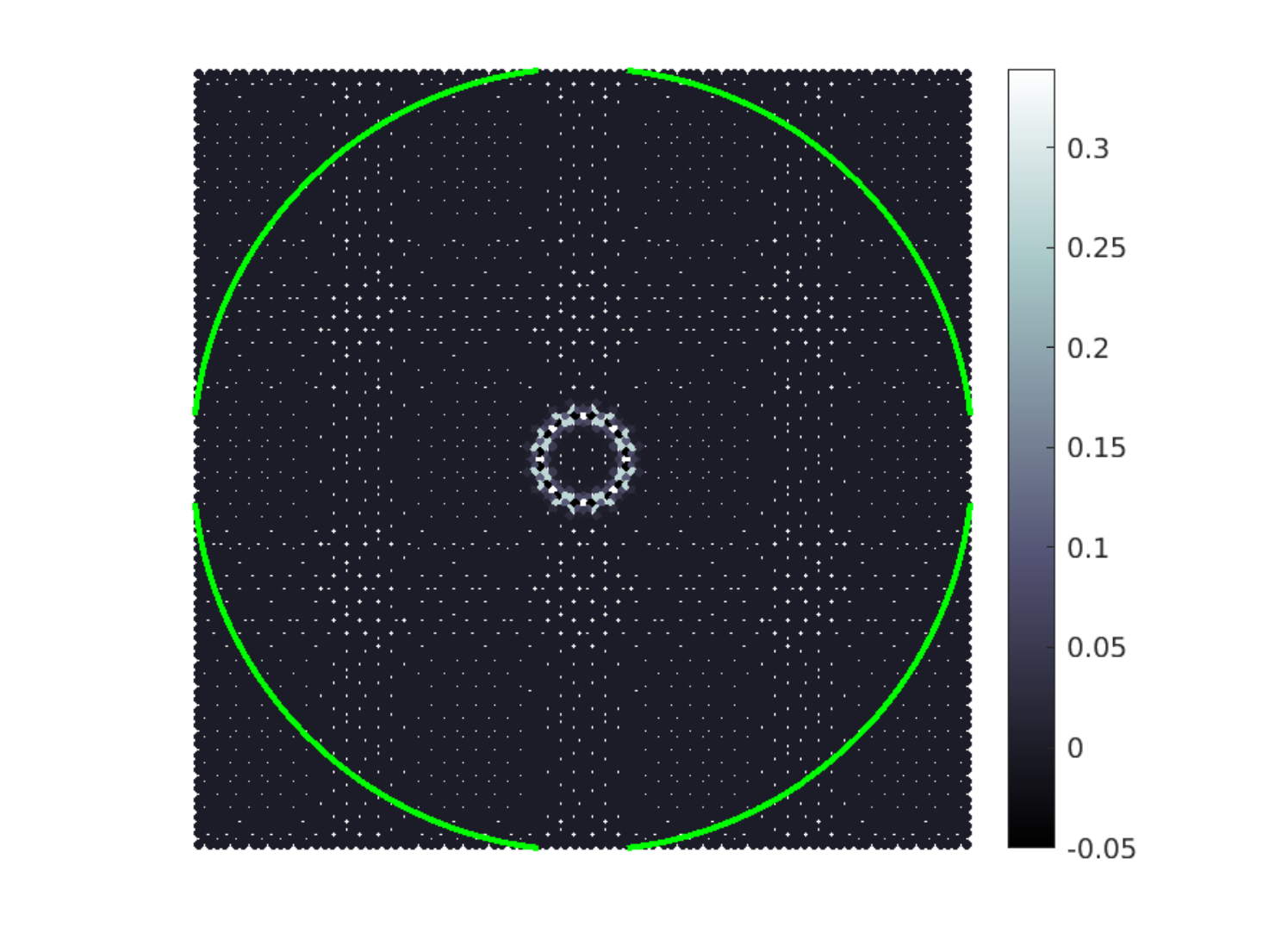}
     \end{overpic}
  \end{minipage}
	\begin{minipage}[b]{0.32\textwidth}
    \vspace{-4mm}\begin{overpic}[width=\textwidth]{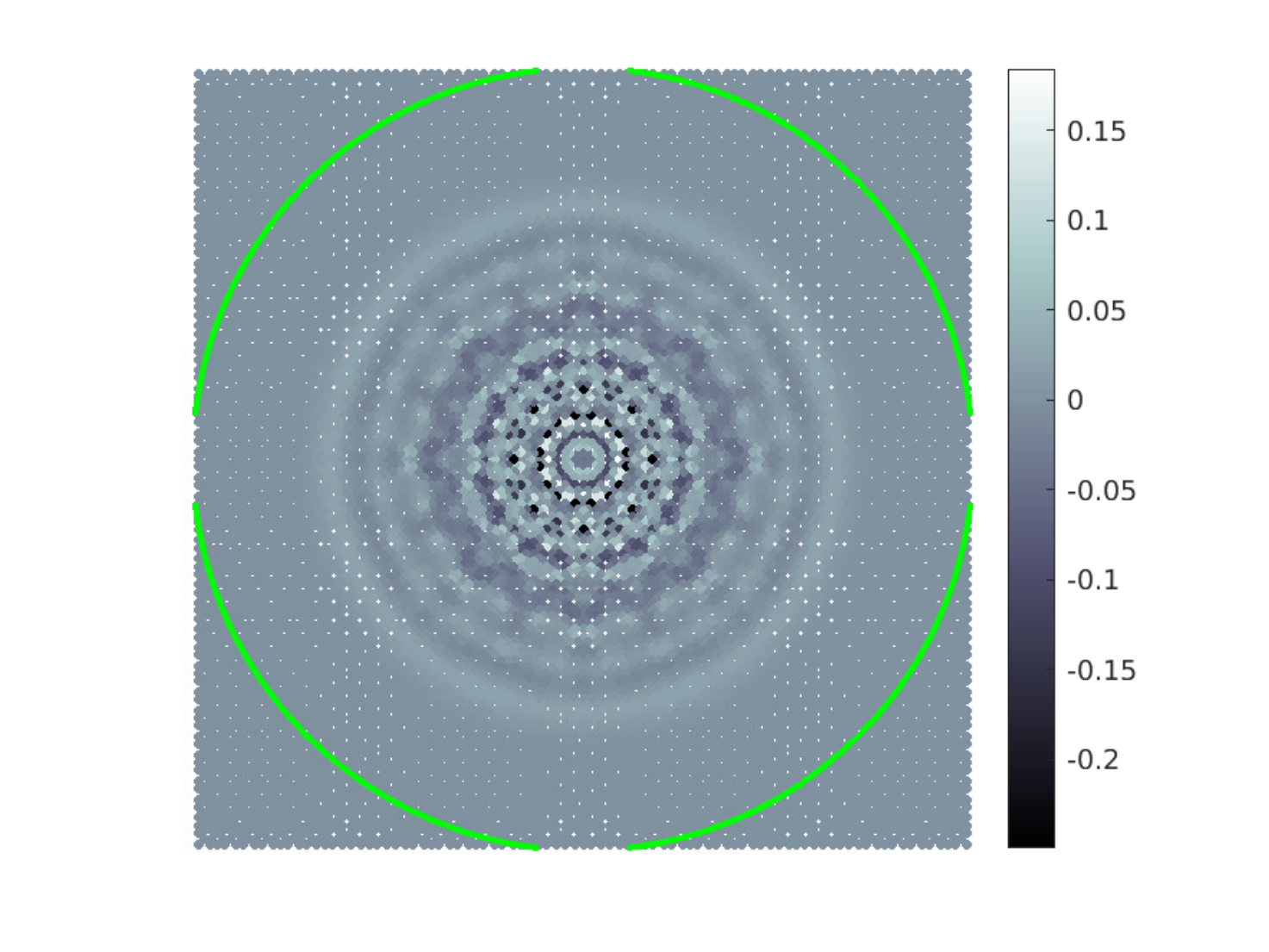}
     \end{overpic}
  \end{minipage}
	\begin{minipage}[b]{0.32\textwidth}
    \vspace{-4mm}\begin{overpic}[width=\textwidth]{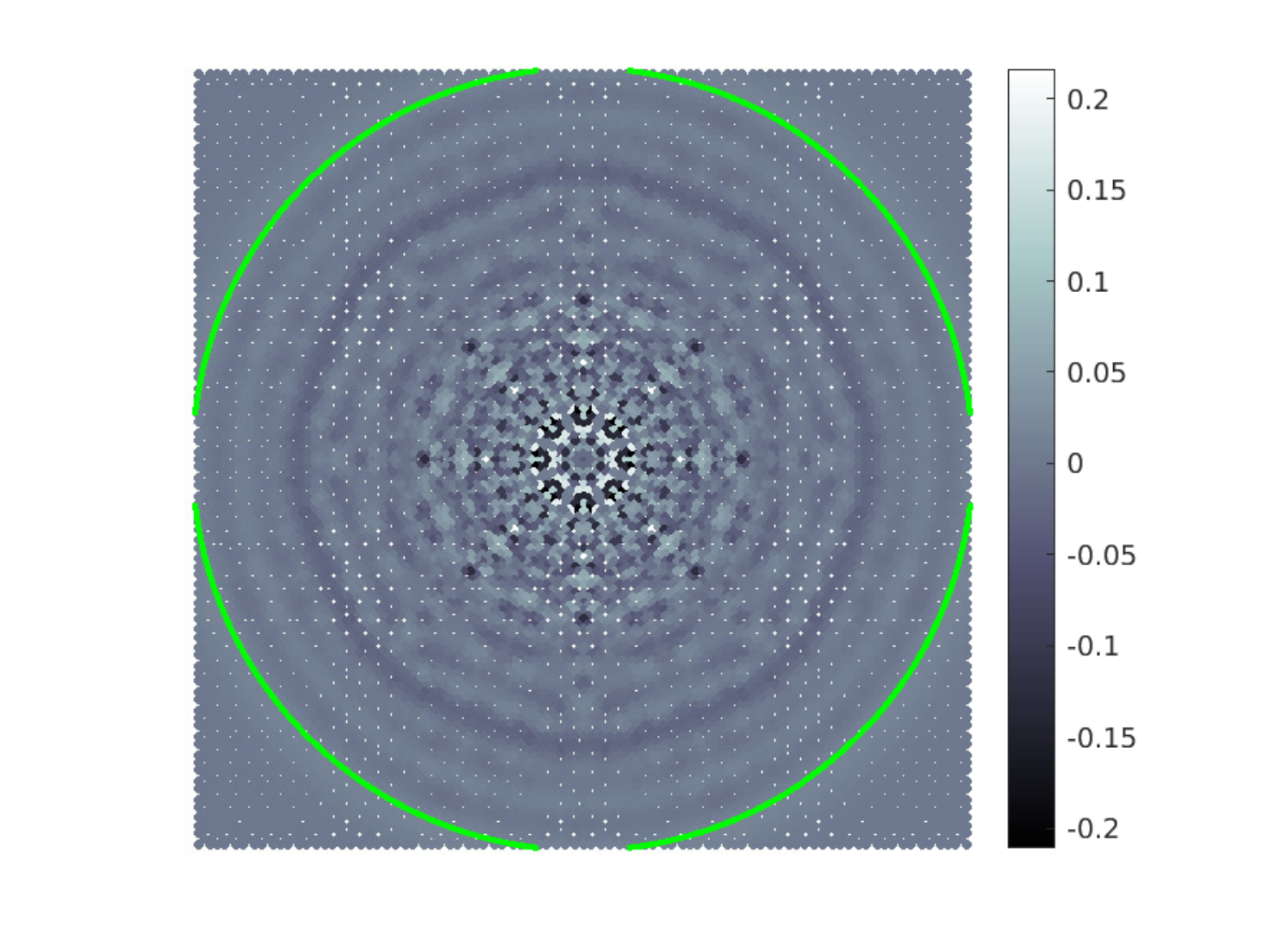}
     \end{overpic}
  \end{minipage}
  \vspace{-4mm}\caption{Top row: $\texttt{log10}(|u(t)|)$ computed for the Schr\"odinger equation at times $t=1$ (left), $t=10$ (middle) and $t=50$ (right). The white circle corresponds to the truncation of the tiling ($10001$ vertices) when computing $u_{\mathrm{FS}}$ for \cref{fig:AB_1}. Bottom row: $u(t)$ computed for the wave equation at times $t=1$ (left), $t=30$ (middle) and $t=50$ (right). The green circle corresponds to the truncation of the tiling when computing $u_{\mathrm{FS}}$ and we have zoomed in compared with the top row for clarity.}\label{fig:AB_2}
\end{figure}

We compute $u$ with an error bounded by $\epsilon=10^{-10}$. Because of the larger truncation size needed to compute the solution at larger times, the total computation time for $u$ for this example was of the order of minutes (with parallelization over quadrature points using 20 CPU cores) as opposed to seconds. The cost per linear solve to compute $R(z,\Delta_{\mathrm{AB}})$ scales as $\mathcal{O}(m^{{3/2}})$ for $m$ basis sites, since the local bandwidth of the infinite matrix grows. The number of basis sites needed depends on the tolerance required and the distance of $z$ to the spectrum. To demonstrate the difficulties of finite-dimensional approaches, we also consider $u_{\mathrm{FS}}$, the solution obtained by direct diagonalization\footnote{This method was chosen to distill the error associated with truncating the AB tiling as opposed to errors when approximating the exponential of a finite matrix via other methods.} of the Galerkin truncation $P_n\Delta_\mathrm{AB}P_n$ for $n=10001$ (chosen to maintain rotational symmetry). \cref{fig:AB_1} (right) shows the difference in norm between the computed $u$ and $u_{\mathrm{FS}}$. The difference is small for small time $t$. However, the difference begins to grow quickly as $u_{\mathrm{FS}}$ becomes inaccurate because of boundary effects. This is demonstrated in \cref{fig:AB_2}, which plots the computed solutions $u$. As $t$ increases, we need more vertices (basis vectors) to capture the solution. The method of this paper allows this to be done automatically in a rigorous and adaptive manner.
 
\subsection{Complex perturbed fractional diffusion equation}\label{complex_diff_example} Our next example demonstrates the results of \cref{analytic_semigp_sec} for $\delta>0$, and also deals with more general Laplace transform inversions. We consider the following equation on $L^2(\mathbb{R})$
\begin{equation}\setlength\abovedisplayskip{6pt}\setlength\belowdisplayskip{6pt}
\label{fractional_diff}
D_t^{\iota}u=u_{xx}+{iu}/{(1+x^2)},\quad 0<\iota\leq 1,
\end{equation}
where $D_t^{\iota}$ denotes Caputo's fractional derivative \cite{mainardi2000mittag}. Such fractional diffusion problems are attracting increasing interest and have many applications \cite{zeng2014crank,diethelm2010analysis}. The solution is computed using the same method as in \cref{analytic_semigp_sec}, but now with the resolvent $(A-zI)^{-1}$ replaced by $(Az^{1-\iota}-z)^{-1}$. We use the Malmquist--Takenaka basis functions and the initial condition $u_0$ from \eqref{variable_diff_example}. \cref{fig:frac_diff} shows solutions computed with an error bound $\epsilon=10^{-12}$ for various $\iota$ and times $t$. Even for this small value of $\epsilon$, the computational times for this example (including forming the linear system, solving for all quadrature points etc.) were at most on the order of seconds on a modest laptop without parallelization. Again, the cost of approximating $(Az^{1-\iota}-z)^{-1}$ scales as $\mathcal{O}(m)$ (using $\mathcal{O}(m)$ basis functions) owing to the banded matrix representation. The size of $m$ needed depends on the point $z$, as well as the required accuracy.

Despite the same initial conditions, the diffusion changes dramatically with $\iota$, with a slower diffusion process occurring for smaller $\iota$ ($\iota<1$ is known as sub-diffusion). When applying domain truncation techniques, it can be challenging, particularly for small $\iota$, to determine a suitable truncation for a given desired accuracy. Similarly, even when using global basis expansions that do not use a domain truncation, the number of basis functions needed to capture the solution can be hard to predict a priori (and one would also want a proof of convergence). Again, the method of this paper overcomes these issues rigorously and adaptively. Another benefit of contour methods for fractional PDEs is the reduced memory requirement compared to time stepping methods, which typically store the history of the solution because of the non-local nature of $D_t^{\iota}$ \cite{li2009space}. Moreover, we avoid having to resolve singularities of the solution as $t\downarrow 0$. Instead, we compute accurate solutions simultaneously over large time intervals using \cref{alg:spec_meas}.

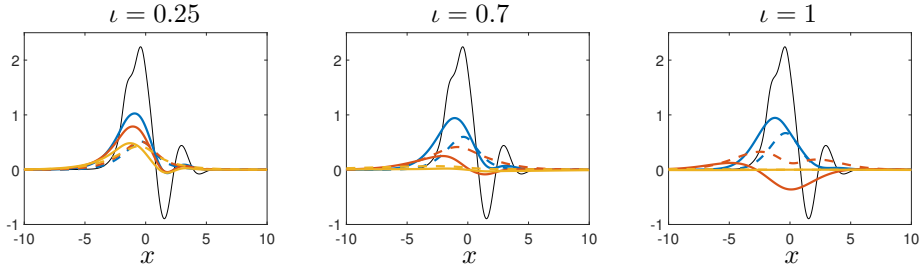
\begin{figure}
  \centering

	\begin{minipage}[b]{0.32\textwidth}
    \begin{overpic}[width=\textwidth]{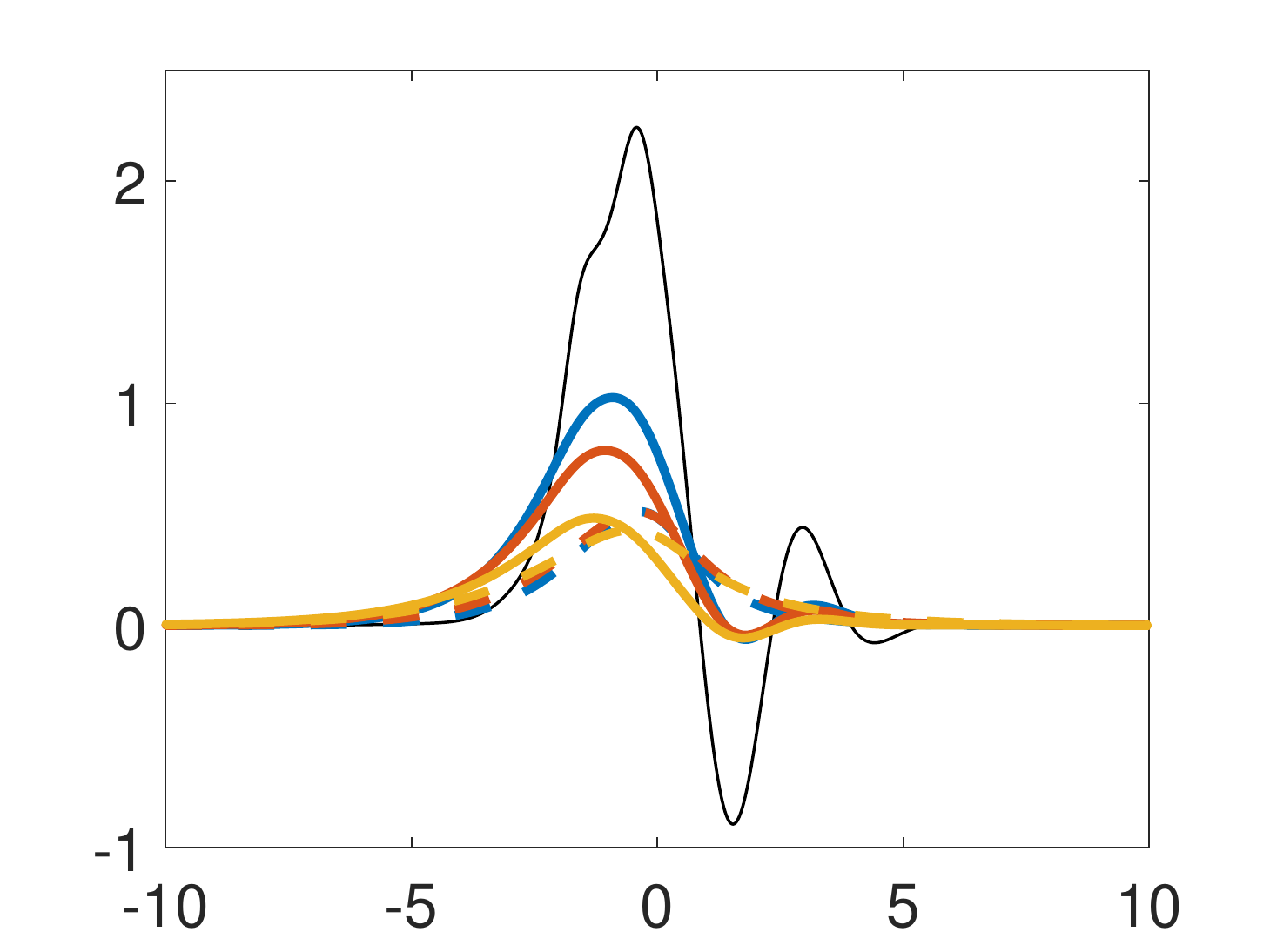}
    \put (40,73) {$\iota=0.25$}
		\put(50,-4) {$x$}
     \end{overpic}
  \end{minipage}
	\begin{minipage}[b]{0.32\textwidth}
    \begin{overpic}[width=\textwidth]{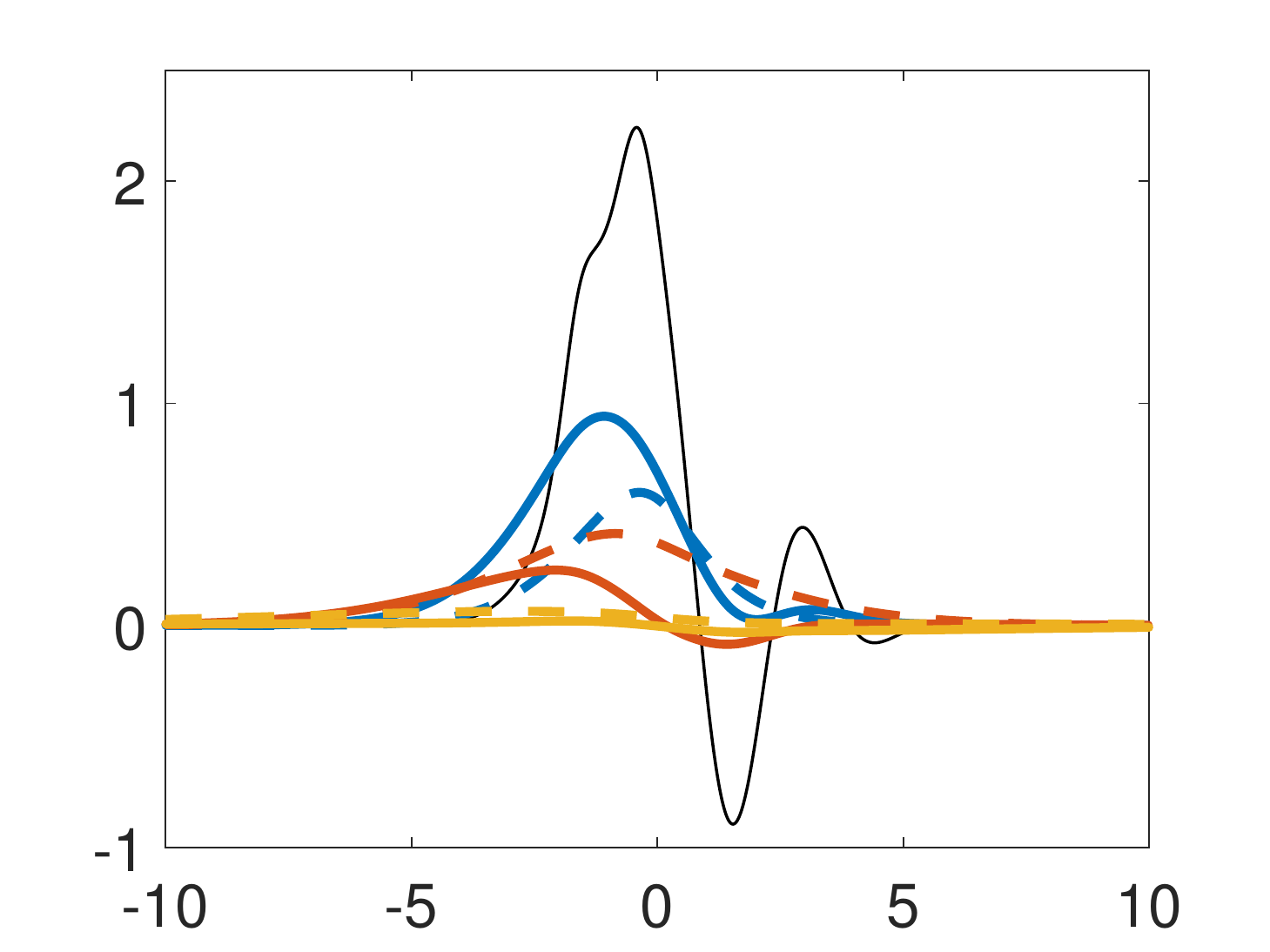}
		\put (40,73) {$\iota=0.7$}
		\put(50,-4) {$x$}
     \end{overpic}
  \end{minipage}
	\begin{minipage}[b]{0.32\textwidth}
    \begin{overpic}[width=\textwidth]{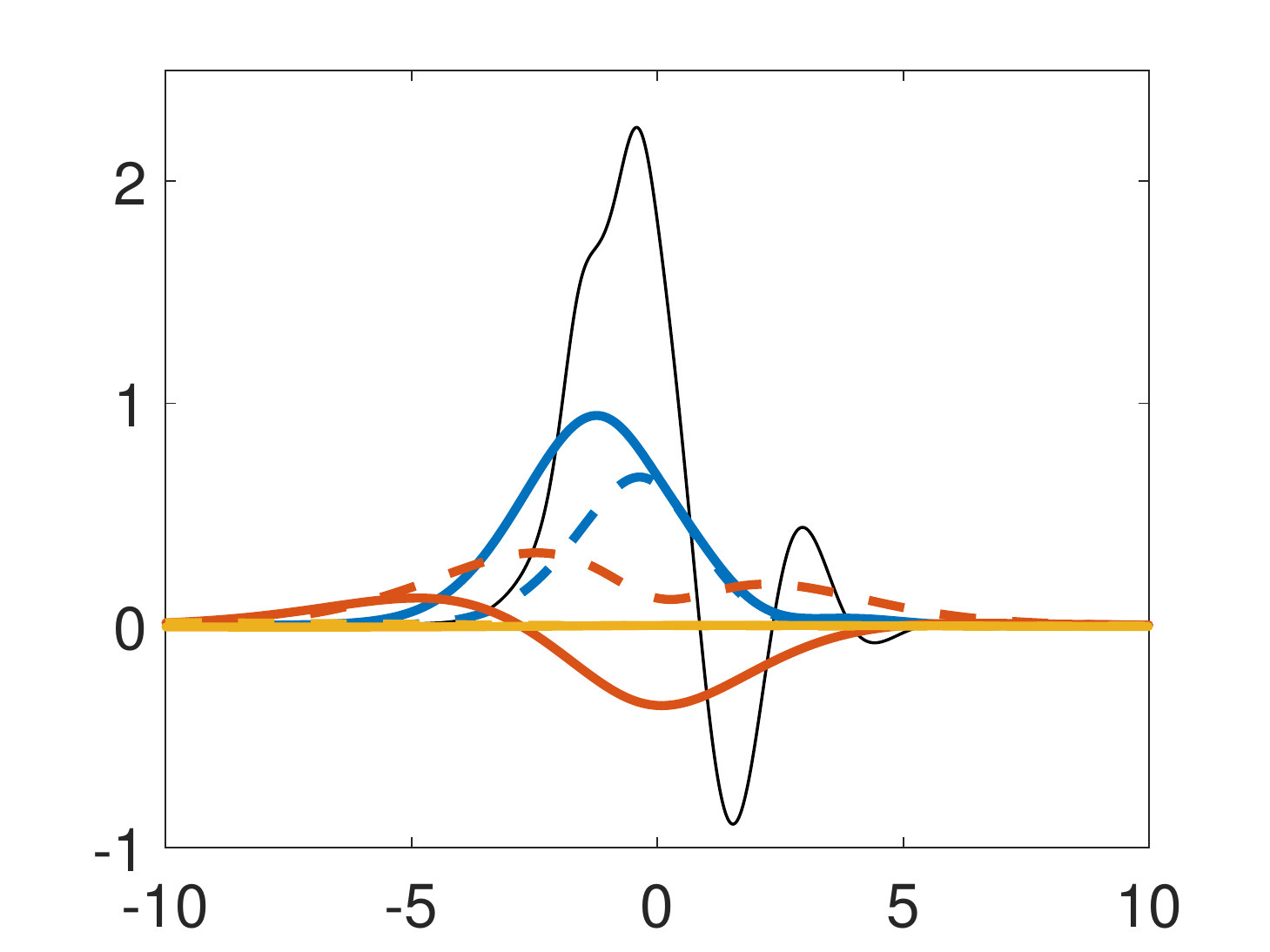}
		\put (42,73) {$\iota=1$}
		\put(50,-4) {$x$}
     \end{overpic}
  \end{minipage}
  \vspace{-1mm}\caption{Solution of \eqref{fractional_diff} for various $\iota$ at $t=1$ (blue), $t=5$ (red) and $t=50$ (yellow). The real parts are shown as solid lines, and the imaginary parts as dashed lines ($u_0$ shown in black).}\label{fig:frac_diff}\vspace{-1mm}
\end{figure}

\subsection{Euler--Bernoulli beam equation with Kelvin--Voigt damping}\label{Elastic_example}

As a final example, we consider the following second-order problem on the interval $[-1,1]$
\begin{equation}\setlength\abovedisplayskip{6pt}\setlength\belowdisplayskip{6pt}\label{beam_example_equation}
\rho(x)u_{tt}=-[a(x)u_{xx}+b(x)u_{xxt}]_{xx},\quad u(\pm 1,t)=u_x(\pm,t)=0,
\end{equation}
with $u(x,0)=(1-x^2)\sin(5\pi x),u_t(x,0)=x$, and the choices
$$\setlength\abovedisplayskip{6pt}\setlength\belowdisplayskip{6pt}
\rho(x)=1+{x}/{2},\quad a(x)=1+{x^3}/{2},\quad b(x)=\tanh(10(x-0.7))+1.01.
$$
The damping function $b$ models the suppression of vibrations of a clamped elastic beam \cite{liu1998spectrum}. This example has a non-normal generator with complex spectrum, and has non-empty continuous spectrum \cite{zhang2011spectrum} despite being posed on a finite interval. The problem is well-posed for $(u_0,u_0')\in\mathcal{H}_0^2([-1,1])\times L^2([-1,1])$, but for simplicity, we measure the error of computed solutions in $L^2([-1,1])$. To solve the shifted linear systems, we use the ultraspherical spectral method \cite{Olver_SIAM_Rev}.\footnote{By using sparse approximations and suitable rectangular truncations, corresponding to bounding tails of expansions of coefficients and solutions in ultraspherical polynomials, we gain error control. See \cite{brehard2018validated} for an interval arithmetic implementation of the ultraspherical spectral method.} \cref{fig:beam_1} shows computed solutions with an error bound $\epsilon=10^{-12}$. Even for this small value of $\epsilon$, the computational times (including forming the linear system, solving for all quadrature points etc.) were at most on the order of seconds on a modest laptop without parallelization. The cost of approximating the resolvent using $m$ basis functions scales as $\mathcal{O}(m)$ owing to the almost banded matrix representation (filled in rows correspond to the boundary conditions). Since $b$ is only significant from zero for $x>0.7$, this region exhibits the anticipated damped behavior, whilst the section $x<0.7$ undergoes almost free vibration.

\begin{figure}\vspace{-0.5mm}
  \centering
  \begin{minipage}[b]{0.48\textwidth}
    \begin{overpic}[width=\textwidth]{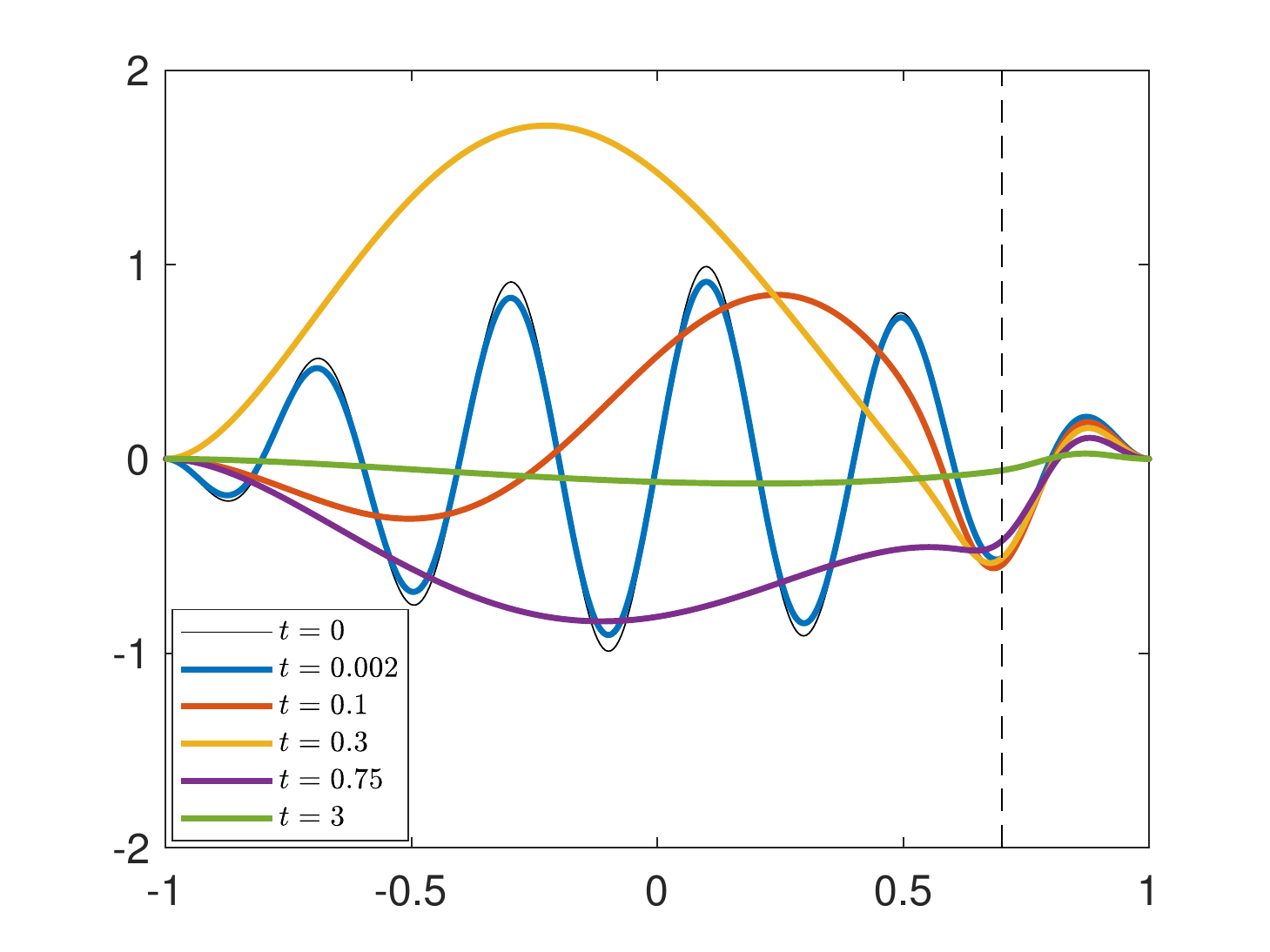}
		\put(50,-1) {$x$}
     \end{overpic}
  \end{minipage}
\vspace{-2.5mm}\caption{Solutions to \eqref{beam_example_equation} computed at different times $t$ with an error bound $\epsilon=10^{-12}$. The dashed line shows the transition region between small and large damping function $b$.}\label{fig:beam_1}\vspace{-1mm}
\end{figure}

\section{Concluding remarks}\label{conc_sect}

We have developed an algorithm that computes semigroups on infinite-dimensional separable Hilbert spaces with explicit and rigorous error control. \cref{alg:l2_semigroup} summarizes the approach. We combine a regularized functional calculus, suitable contours and quadrature, and machinery used to compute the resolvent with error control. We derive results for both the canonical Hilbert space $l^2(\mathbb{N})$ and partial differential operators on unbounded domains. For analytic semigroups, we derive a stable and rapidly convergent scheme. \cref{alg:spec_meas} summarizes this method, which is suitable for infinite-dimensional operators and more general Laplace transform inversions. The result is a fully adaptive and rigorous method, with the flexibility of only requiring solving linear systems with error control. 

There are numerous possible extensions of the current work where error control is an advantageous feature. Our algorithm could be a building block for inhomogeneous, non-autonomous and non-linear evolution equations. For example, by controlling the numerical error of exponentials and similar functions of linearized operators, this could be useful for low regularity integrators \cite{rousset2021general,ostermann2019two,bruned2020resonance}. Other avenues to explore include generalizations of \cref{PDO_thm1} for different classes of coefficients (e.g., singular terms, which are particularly relevant to the Scr\"odinger case in \eqref{cauchy_prob_scrod}) and different choices of basis functions. Another possibility is an error control algorithm using finite element discretizations instead of spectral methods. This extension would be particularly useful for complicated domains. For conforming finite elements with sufficient regularity, it should be possible to reduce the problem to one over $l^2(\mathbb{N})$ (as done in \cref{PDO_thm1}) by computing the needed inner products and modifying our algorithm with a suitable Gramian matrix.

Finally, we point out that contour methods are certainly not the most efficient method for every type of semigroup. This point is reflected by the diverse list of methods in \cref{sec_intro_sec} and the fact that contour methods have historically been used mainly for analytic semigroups. Contour methods do, however, provide a positive answer to the \textit{foundations} questions Q.1 and Q.2 in the introduction and lend themselves readily to an infinite-dimensional ``solve-then-discretize'' approach. In the future, we will explore using other techniques in answering similar foundational questions in infinite dimensions.

\appendix
\section{Auxiliary results}\label{sec:appendix}
We give two results needed in our proofs.

\begin{lemma}
\label{boring}
Let $x\in l^2(\mathbb{N})$ and suppose that we have access to evaluation functions $\{f_{j,m}\}$ as in \eqref{needinlemma}. Then given any $\epsilon>0$, we can compute an approximation $x^{\epsilon}$ with finite support (with respect to the canonical basis) such that $\|x-x^{\epsilon}\|\leq \epsilon$.
\end{lemma}
\begin{proof}
Let $x^\epsilon=\sum_{j=1}^Mf_{j,m}(x)e_j,$ where we choose $M$ and $m$ in the following manner. Clearly $x^{\epsilon}$ has finite support with respect to the canonical basis. We also have that $
\|x-x^{\epsilon}\|^2\leq \sum_{j=1}^M2^{-2m}+\sum_{j>M}|x_j|^2$ and so we must choose $M$ and $m$ so that this bound is less than $\epsilon^2$. Given $M$, we can choose $m$ so that the first term is bounded by $\epsilon^2/2$, so it suffices to choose $M$ so that the second term is bounded by $\epsilon^2/2$. But we have that $\sum_{j>M}|x_j|^2=\langle x,x\rangle -\sum_{j=1}^M |x_j|^2$.
Given the evaluation functions (now making use of the $f_{0,m}$ that approximate $\|x\|^2$), we can compute the right-hand side of this equation to any given accuracy. By doing this for successive $M$, we can compute an $M$ adaptively such that $\sum_{j>M}|x_j|^2\leq\epsilon^2/2$.
\end{proof}

Next, we consider evaluating the resolvent. The following proposition uses an adaptive least-squares approximation $[P_nT^*TP_n]^{-1}P_nT^*x$ and a posteriori bounds on residuals of the corresponding infinite-dimensional linear system.

\begin{proposition}
\label{res_est1}
Consider the setup in \cref{disc_thm_sec}. Given $\epsilon>0$, there exists an algorithm $\Gamma_{\epsilon}$ that when given $(T,x)\in\mathcal{C}(l^2(\mathbb{N}))\times l^2(\mathbb{N})$ with $0\notin\mathrm{Sp}(T)$, uses $\Lambda_1$ (but now for $T$ instead of $A$) to compute a vector $\Gamma_{\epsilon}(T,x)$ such that:
\begin{enumerate}
	\item $\Gamma_{\epsilon}(T,x)$ has finite support with respect to the canonical basis.
	\item For any input, $	\left\|\Gamma_{\epsilon}(T,x)-T^{-1}x\right\|\leq \epsilon\|T^{-1}\|$.
\end{enumerate}
\end{proposition}

\begin{proof}
Let $(T,x)$ denote a suitable input as in the statement of the proposition. Since $0\notin\mathrm{Sp}(T)$, $n=\mathrm{rank} (P_n)=\mathrm{rank} (TP_n)$. Hence we can define the least-squares solution
$$\setlength\abovedisplayskip{6pt}\setlength\belowdisplayskip{6pt}
\widetilde\Gamma_{n}(T,x):=\mathrm{argmin}_{y}\|TP_ny-x\|=[P_nT^*TP_n]^{-1}P_nT^*x.
$$
The space $\mathrm{span}\{e_n:n\in\mathbb{N}\}$ forms a core of $T$. It follows by invertibility of $T$ that given any $\delta>0$, there exists an $m=m(\delta)$ and a $y=y(\delta)$ with $P_my=y$ such that
$$\setlength\abovedisplayskip{6pt}\setlength\belowdisplayskip{6pt}
\|Ty-x\|\leq\delta.
$$
It follows that for all $n\geq m$, $\|T\widetilde\Gamma_{n}(T,x)-x\|\leq \|Ty-x\|\leq \delta$ and hence that
$$\setlength\abovedisplayskip{6pt}\setlength\belowdisplayskip{6pt}
\|\widetilde\Gamma_{n}(T,x)-T^{-1}x\|\leq \delta\|T^{-1}\|.
$$
Since $\delta>0$ was arbitrary, we see that $\widetilde\Gamma_{n}(T,x)$ converges to $T^{-1}x$ as $n\rightarrow\infty$.

For $n,m\in\mathbb{N}$, define the finite matrices
$$\setlength\abovedisplayskip{6pt}\setlength\belowdisplayskip{6pt}
B_n=P_nT^*TP_n,\quad C_{m,n}=P_nT^*P_{m}.
$$
Given the evaluation functions in $\Lambda_1$, we have access to the entries of these matrices to any desired accuracy. It follows that we can compute approximations of $B_n$ and $C_{m,n}$ denoted $\widetilde B_n$ and $\widetilde C_{m,n}$ respectively with
$$\setlength\abovedisplayskip{6pt}\setlength\belowdisplayskip{6pt}
\max\left\{\|B_n-\widetilde B_n\|,\|B_n^{-1}-\widetilde B_n^{-1}\|,\|C_{m,n}-\widetilde C_{m,n}\|\right\}\leq n^{-1}.
$$
We then define
$$\setlength\abovedisplayskip{6pt}\setlength\belowdisplayskip{6pt}
\Gamma_{m,n}(T,x):=\widetilde B_n^{-1}\widetilde C_{m,n} x_{(m)},
$$
where $x_{(m)}=P_mx_{(m)}$ is an approximation of $P_mx$ to accuracy $m^{-1}$. The bounds $n^{-1}$ and $m^{-1}$ are for convenience only. Using the above bounds, we have that
$$\setlength\abovedisplayskip{6pt}\setlength\belowdisplayskip{6pt}
\|\Gamma_{m,n}(T,x)-\widetilde\Gamma_{n}(T,x_{(m)})\|\!\leq\! \|B_n^{-1}-\widetilde B_n^{-1}\|\!\|C_{m,n}\|\!\|x_{(m)}\|+\|\widetilde B_n^{-1}\|\!\|C_{m,n}-\widetilde C_{m,n}\|\!\|x_{(m)}\|,
$$
so that $\Gamma_{m,n}(T,x)$ converges to $T^{-1}x_{(m)}$ as $n\rightarrow\infty$. By construction, $\Gamma_{m,n}(T,x)$ has finite support with respect to the canonical basis. Given $\epsilon>0$, we choose, using \cref{boring}, $m=m(\epsilon)$ such that $\|x-x_{(m)}\|\leq \epsilon/2$. Letting $\Gamma_{m,n}$ denote $\Gamma_{m,n}(T,x)$,
\begin{align}\setlength\abovedisplayskip{6pt}\setlength\belowdisplayskip{6pt}
\|\Gamma_{m,n}-T^{-1}x_{(m)}\|^2&\leq \|T^{-1}\|^2\|T\Gamma_{m,n}-x_{(m)}\|^2\\
&=\|T^{-1}\|^2\left[\|T\Gamma_{m,n}\|^2-2\mathrm{Re}(\langle T\Gamma_{m,n},x_{(m)} \rangle)+\|x_{(m)}\|^2\right].\label{nnafjgv}
\end{align}
Since we have access to approximations of both $\langle Te_j, Te_j \rangle$ (norms of the columns of $T$) and $\langle Te_i, e_j \rangle$ in $\Lambda_1$, we can approximate all of the terms in the squared brackets of \eqref{nnafjgv} to any desired accuracy. We therefore choose $n=n(\epsilon)$ so that $\|T\Gamma_{m,n}-x_{(m)}\|^2\leq \epsilon^2/4$ and set $\Gamma_{\epsilon}(T,x)=\Gamma_{m(\epsilon),n(\epsilon)}(T,x)$ so that
$$\setlength\abovedisplayskip{6pt}\setlength\belowdisplayskip{6pt}
\left\|\Gamma_{\epsilon}(T,x)-T^{-1}x\right\|\!\leq\!\|T^{-1}\|\|x-x_{(m)}\|+\|\Gamma_{m,n}-T^{-1}x_{(m)}\|\!\leq \!\|T^{-1}\|\left(\frac{\epsilon}{2}\!+\!\frac{\epsilon}{2}\right)=\epsilon\|T^{-1}\|.
$$
This bound completes the proof.
\end{proof}

In practice, we approximate $P_mx$, $B_n$ and $C_{m,n}$ using floating-point arithmetic and then apply the above argument to bound the residual. The discretization size is increased adaptively until the specified tolerance has been reached.

\section*{Acknowledgements} I would like to thank Lorna Ayton, Andrew Horning, Arieh Iserles and Alex Townsend for interesting discussions during the completion of this work and helpful comments regarding a draft version of this paper. 

\bibliographystyle{siamplain}
\bibliography{semigroup}

\end{document}